\documentclass[final]{siamart0216}
\usepackage{epsfig}
\usepackage{color}
\usepackage{amscd}
\usepackage{amsmath, amsfonts, amssymb,mathrsfs}
\usepackage{multirow,multicol}
\usepackage{enumerate}
\usepackage{verbatim}
\usepackage{cite}
\usepackage{tikz}
\usetikzlibrary{matrix,arrows}
\usepackage[justification=raggedright]{caption}
\usepackage{subcaption}
\usepackage[export]{adjustbox}
\usepackage{bm}
\usepackage{slashbox,nicefrac}
\usepackage{multirow}

\newcommand{\curl}{{\rm curl}}
\newcommand{\dive}{{\rm div}}
\newcommand{\grad}{{\rm grad}}

\newcommand{\Nedelec}{N\'{e}d\'{e}lec }

\newcommand{\scalar}[2]{\langle #1, #2\rangle} 
\renewcommand{\vec}[1]{\ensuremath{\bm{#1}}}

\newcommand{\bigA}{\mathcal{A}}
\newcommand{\bigAaux}{\mathcal{A}^{\textrm{\tiny aux}}}
\newcommand{\massM}{M}
\newcommand{\matG}{G}
\newcommand{\matK}{K}
\newcommand{\matD}{D}
\newcommand{\matZ}{Z}
\newcommand{\preconW}{\mathcal{W}}
\newcommand{\preconWaux}{\mathcal{W}^{\textrm{\tiny aux}}}
\newcommand{\preconBF}{\mathcal{X}}
\newcommand{\prediag}{\mathcal{D}}
\newcommand{\prelow}{\mathcal{L}}
\newcommand{\preupp}{\mathcal{U}}
\newcommand{\preQ}{Q}
\newcommand{\schurS}{S}

\title{Robust Solvers for Maxwell's Equations with Dissipative Boundary Conditions
\thanks{Submitted \today.}}

\author{J. H. Adler\thanks{Department of Mathematics, Tufts
    University, Medford, MA 02155 (james.adler@tufts.edu,
    \hbox{xiaozhe.hu@tufts.edu}).  The work of J.~Adler was supported in part by the National Science 
Foundation under grant DMS-1216972} \and X. Hu\footnotemark[2] \and
  L. T. Zikatanov\thanks{Department of Mathematics, The Pennsylvania
    State University, University Park, PA 16802 (\hbox{ludmil@psu.edu}). The work of L.~Zikatanov was supported in part by the National Science 
Foundation under grant DMS-1418843 and DMS-1522615.}}

\begin{document}

\maketitle

\begin{abstract}
  In this paper, we design robust and efficient linear solvers for the
  numerical approximation of solutions to Maxwell's equations
  with dissipative boundary conditions.  We consider a structure-preserving finite-element approximation with standard
  N\'{e}d\'{e}lec--Raviart--Thomas elements in space and a
  Crank--Nicolson scheme in time to approximate the electric and
  magnetic fields.

  We focus on two types of block preconditioners.  The first type is
  based on the well-posedness results of the discrete problem.  The
  second uses an exact block factorization of the
  linear system, for which the structure-preserving
  discretization yields sparse Schur complements. We prove
  robustness and optimality of
  these block preconditioners, and provide supporting numerical tests.
\end{abstract}

\begin{keywords}
Maxwell's equations, finite-element method, structure-preserving block
preconditioners, dissipative boundary conditions.
\end{keywords}

\begin{AMS}
65M60, 35Q61, 65Z05, 65F08, 65F10
\end{AMS}
 
\pagestyle{myheadings}
\thispagestyle{plain}
\markboth{\sc Adler, Hu, Zikatanov}{\sc Robust Solvers for ADS}

\section{Introduction}
In this paper, we consider Maxwell's system of partial differential
equations (PDEs) with dissipative boundary conditions, also known as
impedance boundary conditions.  Let $\mathcal{O}$ be a bounded,
connected domain, $\mathcal{O}\subset \mathbb{R}^{3}$ and consider
Maxwell's equations in the exterior of $\overline{\mathcal{O}}$, that
is, in $\mathbb{R}^3\setminus \overline{\mathcal{O}}$:
\begin{eqnarray}
\vec{B}_t + \curl\, \vec{E} &=& 0,  \label{eq:BcE0}\\
\varepsilon \vec{E}_t - \curl\, \mu^{-1} \vec{B} &=& -\vec{j}, \label{eq:EcB0}\\ 
\dive\, \varepsilon \vec{E}& =& 0\label{eq:divD0},\\
\dive\, \vec{B}& =& 0 \label{eq:divB0}.
\end{eqnarray}
Here, $\varepsilon$ is the permittivity of the medium, $\mu$ is the
permeability, and $\vec{j}$ is the known current density of the system
satisfying $\dive\, \vec{j}=0$.  We assume that the computational domain, $\Omega=\mathcal{S}\setminus \overline{\mathcal{O}}$, is bounded,
where $\mathcal{S}$ is a ball in $\mathbb{R}^{3}$ with sufficiently
large radius that contains $\mathcal{O}$. 
The system
\eqref{eq:BcE0}-\eqref{eq:divB0} is subject to a dissipative boundary
condition:
\begin{equation}\label{eq:boundary-condition}
(1+\gamma )\vec{E}_{\textrm{tan}}=  -\vec{n}\wedge\vec{B}, \quad \mbox{on} \ \Gamma_i. 
\end{equation} 
In this setting, $ \Gamma_{i}=\partial\Omega\cap \partial \mathcal{O}$
and
$\vec{F}_\textrm{tan} = \vec{F}-\langle
\vec{F},\vec{n}\rangle\vec{n}$,
for a vector-valued function $\vec{F}$.  On the rest of the boundary,
$\Gamma_{o}=\partial\Omega\setminus \Gamma_i$, we have
essential (Dirichlet-type) boundary conditions. 
For symmetric hyperbolic systems, such problems have been investigated
for several decades starting with the work of
Majda~\cite{1974MajdaA-aa,1976MajdaA-aa} and later in the works by
Colombini, Petkov, and Rauch on Maxwell's
equations~\cite{2011ColombiniF_PetkovV_RauchJ-aa,2014ColombiniF_PetkovV_RauchJ-aa,2013PetkovV-aa}.  We note that the
boundary conditions considered in the model problem pertain to
obstacles more general than a perfect conductor. Of course, all of the
constructions in this paper also apply to a perfectly-conducting
obstacle (i.e., for
the case of essential boundary conditions on the entire boundary).

In the following, we develop efficient solvers based on block
factorizations of structure-preserving discretizations of Maxwell's
equations, \eqref{eq:BcE0}--\eqref{eq:divB0}, with dissipative
boundary conditions, \eqref{eq:boundary-condition}.  The goal is to
efficiently solve the full time-dependent problem, uniformly with
respect to physical and discretization parameters.  The
finite-element discretization that we use is described in
\cite{2013AdlerJ_PetkovV_ZikatanovL-aa} with further details
included below. A serious bottleneck in the simulations based on
this discretization, however, was the computational work needed for
the solution of the resulting linear systems at each time step.  As
shown later, both theoretically and via numerical experiments, this
issue is resolved by efficient and robust preconditioning
techniques proposed here.

Block preconditioners are often used for coupled systems, especially
those of saddle-point type (see e.g.,
\cite{Benzi.M;Golub.G2005,Benzi.M;Golub.G;Liesen.J2005,Elman.H;Silvester.D;Wathen.A.2005a,klawonn1998block,
  Loghin.D;Wathen.A2004,Mardal.K;Winther.R2004,Rusten.T;Winther.R1992,Schoeberl2007,Vassilevski.P2008}). Such
preconditioners usually decouple the problems at the preconditioning
stage and convert complicated systems into several simpler problems
for which efficient solvers are either known or easier to
construct. In general, there are two approaches to construct these
types of preconditioners: \emph{analytic} and \emph{algebraic}. The
\emph{analytic} approach constructs the preconditioners by studying
the mapping properties of the differential operators between
appropriate Sobolev spaces.
Prominent examples in this direction are the works of K.~Mardal and
R.~Winther~\cite{Mardal.K;Winther.R2010,Mardal.K;Winther.R2004}, who
developed a class of robust preconditioners for parameter-dependent
problems, such as convection-dominated systems and the time-dependent
Stokes equations.  On the other hand, the \emph{algebraic} approach
aims at constructing preconditioners based on a block decomposition
(or factorization) of the discretized equations. These factorizations
can be very general, but they inevitably involve systems with Schur
complements, which in turn require special approximations.  Examples
of applications include magnetohydrodynamics, where such approximate block
factorization preconditioners have been developed
\cite{Cyr.E;Shadid.J;Tuminaro.R.2013a,2006ElmanH_HowleV_ShadidJ_ShuttleworthR_TuminaroR-aa,Phillips.E;Elman.H;Cyr.E;Shadid.J;Pawlowski.R.2014a}.

In this paper, we present two types of block preconditioners based on these two approaches. For the
analytical approach, we prove the well-posedness of the discrete
problem in appropriate Sobolev spaces equipped with weighted norms.
This allows us to
achieve robustness of the linear solvers with respect to the physical and discretization
parameters of the system.  We then apply the framework from 
\cite{Loghin.D;Wathen.A2004} and \cite{Mardal.K;Winther.R2010} and
construct a family of block diagonal preconditioners, which are
isomorphisms between the same pair of Sobolev spaces. The
action of any such preconditioner corresponds to a decoupled
problem and is computed efficiently.

For the algebraic approach, we derive an exact block factorization of the
resulting linear systems.  In general, this may lead to an inefficient
method, because it requires computing the action of the inverses of
the corresponding Schur complements.  These, typically, are full
matrices of size comparable to the size of the original problem. In
the case of the discretized Maxwell's equations, however, we deal with
special linear systems resulting from finite-element spaces that are
part of a deRham complex.  As a result, we are able to prove that the
Schur complements needed to compute the action of the algebraic
preconditioner are sparse and this action is carried out with an
optimal computational cost.


The paper is organized as follows.  In Section~\ref{sec:preliminary},
we introduce notation and definitions for Maxwell's equations.  The
structure-preserving discretization is then reviewed in
Section~\ref{sec:fem}, and in Section~\ref{sec:solver}, we introduce and
analyze the analytic and algebraic block preconditioners.  Finally, in
Section~\ref{sec:numerics}, we present numerical experiments
illustrating the effectiveness and robustness of the proposed
preconditioners. Concluding remarks and a discussion of future work are given in
Section~\ref{sec:conclude}.

\section{Preliminaries} \label{sec:preliminary}
We use $(\cdot, \cdot)$ and $\| \cdot \|$ to denote the standard $L^2(\Omega)$
inner product and norm on a domain, $\Omega$,
\[
\scalar{u}{v} = \int_{\Omega} u \cdot v \, \mathrm{d}x \text{     and
   } \ \| u \| = \sqrt{\scalar{u}{u}}.
\]
With a slight abuse of notation, we use $L^2(\Omega)$ to denote both
the scalar and vector $L^2$ space. Additionally, we assume that both
$\varepsilon$ and $\mu$ are positive continuous functions only
depending on $x \in \Omega$, inducing weighted $L^2$ norms,
\[
\| \vec{u} \|_{\varepsilon}^2 = \scalar{\varepsilon \vec{u}}{\vec{u}} \ \text{and} \ \| \vec{u} \|_{\mu^{-1}}^2 = \scalar{\mu^{-1} \vec{u}}{ \vec{u}}.
\] 

Next, given a Lipschitz
domain, $\Omega$, and a differential operator, $\mathfrak{D}$, we use a standard notation for the following spaces
\[
H({\mathfrak{D}})= \{v\in L^2(\Omega), {\mathfrak{D}} v\in
L^2(\Omega)\},
\]
with the associated graph norm, $\|u\|_{\mathfrak D}^2 = \|u\|^2 +
\|\mathfrak{D} u\|^2$ (e.g. $H(\grad) = H^1(\Omega)$).  Then, we introduce the following spaces (the first one for scalar functions and the rest for vector-valued functions): 
\begin{eqnarray*}
  H_0(\grad) = H^1_0(\Omega) & = & \{ v\in H^1(\Omega)\quad\mbox{such
    that}\quad v\big|_{\partial\Omega}=0\},\\
  H_{\textrm{imp}}(\curl) & = & \{\vec{v}\in H(\curl)\quad\mbox{such
    that}\quad \vec{v}\wedge \vec{n}\big|_{\Gamma_o}=0\},\\
 H_{\textrm{imp}}(\dive) & = & \{ \vec{v}\in H(\dive)\quad\mbox{such that}\quad \scalar{\vec{v}}{\vec{n}}\big|_{\Gamma_o}=0\}, \\
 H^0_{\textrm{imp}}(\dive) & = & \{ \vec{v} \in H_{\textrm{imp}}(\dive), \quad \mbox{such that}\quad \dive \ \vec{v} = 0 \}.
\end{eqnarray*}
More details on the construction of these spaces is found in
\cite{2013AdlerJ_PetkovV_ZikatanovL-aa}.  Finally, for the time-dependent problem considered here, the relevant function spaces
are 
\begin{eqnarray*}
H_0(\grad;t) & = & \{ v(t,\cdot)\in H_0^1(\Omega) \quad\mbox{for all $t\ge 0$}\},\\
H_{\textrm{imp}}(\curl;t) & = & \{ \vec{v}(t,\cdot)\in
H_{\textrm{imp}}(\curl), \quad\mbox{for all $t\ge 0$}\},\\
H_{\textrm{imp}}(\dive;t) & = & \{ \vec{v}(t,\cdot)\in H_{\textrm{imp}}(\dive), \quad\mbox{for all $t\ge 0$}\}.
\end{eqnarray*}

With this notation, following \cite{2013AdlerJ_PetkovV_ZikatanovL-aa},
we introduce an auxiliary variable, $p$, associated with the
divergence-free constraint of $\vec{E}$ and get the following
variational problem:
~\\\\
Find $(\vec{B},\vec{E},p)
\in H_\textrm{imp}(\dive;t)\times H_{\textrm{imp}}(\curl;t)\times H_0(\grad;t)$,
such that for all 
$(\vec{C},\vec{F},q) \in H_\textrm{imp}(\dive)\times
H_{\textrm{imp}}(\curl)\times H_0^1(\Omega)$ and for all $t>0$,
\begin{eqnarray}
\scalar{ \mu^{-1} \vec{B}_t}{\vec{C}} + \scalar{ \mu^{-1}\curl\,
  \vec{E}}{\vec{C}} &= 0,\label{mixed-form-1}\\
\scalar{ \varepsilon \vec{E}_t}{\vec{F}} + \scalar{ \varepsilon\, \grad  \ p}{\vec{F}}
- \scalar{ \mu^{-1} \vec{B}}{\curl\, \vec{F}} + (1+\gamma)\int_{\Gamma_{i}}\scalar{\vec{E}_\textrm{tan}}{\vec{F}_\textrm{tan}}
&= - (\vec{j},\vec{F}) ,\label{mixed-form-2}\\
\scalar{p_t}{q} - \scalar{\varepsilon \vec{E}}{\grad \ q} &= 0.\label{mixed-form-3}
\end{eqnarray}
At $t=0$, the following initial conditions are needed,
\begin{equation}\label{eq:initial-condition}
\vec{E}(0,\vec{x})=\vec{E}_0(\vec{x}),\quad
\vec{B}(0,\vec{x})=\vec{B}_0(\vec{x}),\quad
 p(0,\vec{x})=0.
\end{equation}
In \cite{2013AdlerJ_PetkovV_ZikatanovL-aa}, it was shown that the
above variational problem preserves the divergence of the magnetic
field, $\bm{B}$, strongly and the divergence of the electric field,
$\bm{E}$, weakly, if the initial conditions and right-hand side
satisfy certain conditions.  We discuss this further in the
following section.

\section{Finite-Element Discretization} \label{sec:fem}
Going forward, we consider a structure-preserving discretization of
\eqref{mixed-form-1}-\eqref{mixed-form-3} and discuss the
well-posedness of the linear system obtained at each time step.  Such
analysis is crucial for developing the block preconditioners discussed
in Section \ref{sec:solver}.   

For the temporal discretization, we adopt a Crank-Nicolson scheme.
Crank-Nicholson is an example of a second-order symplectic time-stepping method,
which is capable of preserving the
discrete energy of the system.  These types of schemes are important
for guaranteeing that the asymptotic behavior is captured. If needed, 
higher-order symplectic methods can be used
\cite{1985FengK-aa,1986FengK-aa,1990FengK_WuH_QinM-aa,2010FengK_QinM-aa}.

Spatially, we consider standard finite-element spaces.  For the magnetic field $\vec{B}$, we use the Raviart-Thomas element denoted by $H_{h,\textrm{imp}}(\dive) \subset H_{\textrm{imp}}(\dive)$.  For the electric field $\vec{E}$, we use the \Nedelec element denoted by $H_{h,\textrm{imp}}(\curl) \subset H_{\textrm{imp}}(\curl)$.  Finally, we use standard Lagrange finite elements for the auxiliary unknown, $p$, and denote the space by $H_{h,0}(\grad) \subset H_0(\grad)$.  These choices of finite-element spaces satisfy the following exact sequence, which results in a structure-preserving discretization: 
\begin{equation}\label{def:exact}
\begin{CD}
H_{h,0}(\grad) @>{\grad}>> H_{h,\textrm{imp}}(\curl) @>{\curl}>> H_{h,\textrm{imp}}(\dive) @>{\dive}>> L^2_h,
\end{CD}
\end{equation}
where $L^2_h$ is the corresponding piecewise polynomial subspace of $L^2(\Omega)$.

Thus, the full discretization of Maxwell's equation is: 
~\\\\Find $(\vec{B}_h, \vec{E}_h, p_h) \in H_{h, \textrm{imp}}(\dive) \times H_{h, \textrm{imp}}(\curl) \times H_{h,0}(\grad)$, such that for all $(\vec{C}_h, \vec{F}_h, q_h) \in H_{h, \textrm{imp}}(\dive) \times H_{h, \textrm{imp}}(\curl) \times H_{h,0}(\grad)$,
\begin{eqnarray}
\scalar{ \mu^{-1} \frac{ \vec{B}_h^n -\vec{B}_h^{n-1} }{\tau} }{\vec{C}_h} + \scalar{ \mu^{-1}\curl\, \frac{ \vec{E}^{n}_h + \vec{E}_h^{n-1}  }{2}}{\vec{C}_h} &= 0,\label{discrete-mixed-form-1}\\
\scalar{ \varepsilon \frac{ \vec{E}_h^n - \vec{E}_h^{n-1} }{\tau}}{\vec{F}_h} + \scalar{ \varepsilon \, \frac{\grad  \, p_h^n + \grad \, p_h^{n-1}}{2}}{\vec{F}_h}
- \scalar{ \mu^{-1} \frac{\vec{B}_h^n + \vec{B}_h^{n-1} }{2} }{\curl\, \vec{F}_h} + \nonumber \\
 \qquad \qquad (1+\gamma)\int_{\Gamma_{i}}\scalar{\frac{ \vec{E}^n_{h,\textrm{tan}} + \vec{E}^{n-1}_{h,\textrm{tan}}}{2} }{\vec{F}_{h,\textrm{tan}}} 
= - ( \frac{\vec{j}^n + \vec{j}^{n-1}}{2},\vec{F}_h) ,\label{discrete-mixed-form-2}\\
\scalar{\frac{p_h^n - p_h^{n-1}}{\tau}}{q_h} - \scalar{\varepsilon \frac{ \vec{E}_h^n + \vec{E}_h^{n-1} }{2} }{\grad \, q_h} &= 0,\label{discrete-mixed-form-3}
\end{eqnarray}
with suitable initial conditions,
\begin{equation}
\vec{B}_h^0 = \Pi_{h}^{\dive} \vec{B}_0, \quad \vec{E}_h^0 = \Pi_h^{\curl} \vec{E}_0, \quad p_h^0 = 0.
\end{equation}
Here, the superscripts indicate the time step and $\Pi_h^{\dive}$ and
$\Pi_h^{\curl}$ are the canonical interpolations for $H_{h,
  \textrm{imp}}(\dive)$ and $H_{h, \textrm{imp}}(\curl)$.  This
discretization is structure-preserving, since it preserves the divergence of $\bm{B}$ strongly and the divergence of $\bm{E}$ weakly at the discrete level (as long as the initial conditions and right-hand side are discretized properly).  We refer to \cite{2013AdlerJ_PetkovV_ZikatanovL-aa} for details.

\subsection{Well-posedness}
For simplicity, we drop the subscript $h$ and superscript $n$, and
move all terms involving the previous time step to the right-hand
side.  Thus, the full discretization is stated as follows:
~\\\\ 
Find $(\vec{B}, \vec{E}, p) \in H_{h, \textrm{imp}}(\dive) \times H_{h, \textrm{imp}}(\curl) \times H_{h,0}(\grad)$, such that for all $(\vec{C}, \vec{F}, q) \in H_{h, \textrm{imp}}(\dive) \times H_{h, \textrm{imp}}(\curl) \times H_{h,0}(\grad)$,
\begin{eqnarray}
&& \qquad  \frac{ 2}{\tau}  \scalar{ \mu^{-1} \vec{B} }{\vec{C}} + \scalar{ \mu^{-1}\curl\, \vec{E}}{\vec{C}} = (\vec{g}_{\vec{B}},\vec{C}),\label{discrete-mixed-form-B}\\
&& \qquad \frac{2}{\tau} \scalar{ \varepsilon \vec{E}}{\vec{F}} + \scalar{ \varepsilon \, \grad  \, p}{\vec{F}}
- \scalar{ \mu^{-1} \vec{B} }{\curl\, \vec{F}} + (1+\gamma)\int_{\Gamma_{i}}\scalar{  \vec{E}_{\textrm{tan}}}{\vec{F}_{\textrm{tan}}} = ( \vec{g}_{\vec{E}} ,\vec{F}) ,\label{discrete-mixed-form-E}\\
&& \qquad \frac{2}{\tau}\scalar{p}{q} - \scalar{\varepsilon \vec{E} }{\grad \, q} = (g_p,q),\label{discrete-mixed-form-p}
\end{eqnarray}
where the dual functionals on the right-hand side are defined as
\begin{align*}
(\vec{g}_{\vec{B}},\vec{C}) &= \frac{2}{\tau}\scalar{\mu^{-1} \vec{B}_h^{n-1}}{\vec{C}} - \scalar{\mu^{-1} \curl \
\vec{E}_h^{n-1}}{\vec{C}},\\
(\vec{g}_{\vec{E}},\vec{F}) &= \frac{2}{\tau} \scalar{\varepsilon\vec{E}_h^{n-1}}{\vec{F}} - \scalar{\varepsilon \,
\grad \, p_h^{n-1}}{\vec{F}} + \scalar{\mu^{-1}
\vec{B}_h^{n-1}}{\curl\,\vec{F}} \\
&- (1+\gamma) \int_{\Gamma_{i}}\scalar{ 
\vec{E}^{n-1}_{h,\textrm{tan}}}{\vec{F}_{\textrm{tan}}} - \scalar{\vec{j}^n+\vec{j}^{n-1}}{\vec{F}},\\
(g_p,q) &= \frac{2}{\tau} \scalar{p_h^{n-1}}{q} + \scalar{\varepsilon
  \vec{E}_h^{n-1}}{\grad \, q}.
\end{align*}  

Following the ideas in \cite{Hu.K;Ma.Y;Xu.J.2014a} and
\cite{Hu.K;Hu.X;Xu.J;Ma.Y.2014a}, in order to analyze the
well-posedness of
\eqref{discrete-mixed-form-B}-\eqref{discrete-mixed-form-p}, we
analyze the following auxiliary problem first:
~\\\\
Find
$(\vec{B}, \vec{E}, p) \in H_{h, \textrm{imp}}(\dive) \times H_{h,
  \textrm{imp}}(\curl) \times H_{h,0}(\grad)$,
such that for all
$(\vec{C}, \vec{F}, q) \in H_{h, \textrm{imp}}(\dive) \times H_{h,
  \textrm{imp}}(\curl) \times H_{h,0}(\grad)$,
 \begin{eqnarray}
&&\qquad  \frac{ 2}{\tau}  \scalar{ \mu^{-1} \vec{B} }{\vec{C}} +
\scalar{\mu^{-1} \curl\, \vec{E}}{\vec{C}} + \scalar{ \dive\ \vec{B}}{\dive \ \vec{C}}= (\vec{g}_{\vec{B}},\vec{C}),\label{aux-mixed-form-B}\\
&& \qquad \frac{2}{\tau} \scalar{ \varepsilon \vec{E}}{\vec{F}} + \scalar{ \varepsilon \,  \grad \, p}{\vec{F}}
- \scalar{ \mu^{-1} \vec{B} }{\curl\, \vec{F}} + (1+\gamma)\int_{\Gamma_{i}}\scalar{  \vec{E}_{\textrm{tan}}}{\vec{F}_{\textrm{tan}}} = ( \vec{g}_{\vec{E}} ,\vec{F}) ,\label{aux-mixed-form-E}\\
&& \qquad \frac{2}{\tau}\scalar{p}{q} - \scalar{\varepsilon \vec{E} }{\grad \, q} = (g_p,q).\label{aux-mixed-form-p}
\end{eqnarray}
Since $\dive\,\vec{B} = 0$, the mixed formulations \eqref{discrete-mixed-form-B}-\eqref{discrete-mixed-form-p} and \eqref{aux-mixed-form-B}-\eqref{aux-mixed-form-p} are equivalent if 
$\vec{g}_{\vec{B}} \in \left ( H^0_{h,\textrm{imp}}(\dive)\right )'$.  Thus, the
well-posedness of
\eqref{discrete-mixed-form-B}-\eqref{discrete-mixed-form-p} follows
directly from the well-posedness of
\eqref{aux-mixed-form-B}-\eqref{aux-mixed-form-p}.  

Introducing the following bilinear form,
\begin{align}\label{def:bilinear-L}
a(\vec{B}, \vec{E}, p; \vec{C}, \vec{F}, q ) &:= \frac{ 2}{\tau}
\scalar{ \mu^{-1} \vec{B} }{\vec{C}} + \scalar{ \mu^{-1}\curl\, \vec{E}}{\vec{C}} + \scalar{ \dive\, \vec{B}}{\dive \, \vec{C}}  \nonumber \\ 
& \ + \frac{2}{\tau} \scalar{ \varepsilon \vec{E}}{\vec{F}} + \scalar{ \varepsilon \,  \grad \, p}{\vec{F}} - \scalar{ \mu^{-1} \vec{B} }{\curl\, \vec{F}} + (1+\gamma)\scalar{  \vec{E}}{\vec{F}}_{\Gamma_i}  \\
& \ + \frac{2}{\tau}\scalar{p}{q} - \scalar{\varepsilon \vec{E} }{\grad \, q}, \nonumber
\end{align}
and the following weighted norms,
\begin{align}
\| \vec{B} \|_{\dive}^2 &:= \frac{2}{\tau} \| \vec{B} \|^2_{\mu^{-1}} + \| \dive \, \vec{B} \|^2, \\
\| \vec{E} \|_{\curl}^2 & := \frac{2}{\tau} \| \vec{E} \|^2_{\varepsilon} + \frac{\tau}{2} \| \curl \, \vec{E} \|^2_{\mu^{-1}} + (1 + \gamma) \| \vec{E} \|^2_{\Gamma_i}, \\
\| p \|^2_{\grad} &:= \frac{2}{\tau} \| p \|^2 + \frac{\tau}{2} \| \grad \, p \|^2_{\varepsilon},
\end{align}
we have the following theorem, which shows that \eqref{aux-mixed-form-B}-\eqref{aux-mixed-form-p} is well-posed. 

\begin{theorem}\label{thm:well-posed-aux}
Let $\vec{V}_h: =  H_{h, \textrm{imp}}(\dive) \times H_{h, \textrm{imp}}(\curl) \times H_{h,0}(\grad)$.
The bilinear form defined by \eqref{def:bilinear-L} satisfies the following inf-sup condition, 
\begin{equation} \label{ine:inf-sup}
\sup_{ 0 \neq (\vec{C}, \vec{F}, q) \in \vec{V}_h }\frac{a(\vec{B}, \vec{E}, p; \vec{C}, \vec{F}, q ) }{ \left( \| \vec{C} \|^2_{\dive} + \| \vec{F} \|^2_{\curl} + \| q \|^2_{\grad}  \right)^{1/2} } \geq \frac{1}{4} \left(  \| \vec{B} \|^2_{\dive} + \| \vec{E} \|^2_{\curl} + \| p \|^2_{\grad}  \right)^{1/2},
\end{equation}
and is bounded, 
\begin{equation}\label{ine:bounded}
a(\vec{B}, \vec{E}, p; \vec{C}, \vec{F}, q ) \leq C \left(  \| \vec{B} \|^2_{\dive} + \| \vec{E} \|^2_{\curl} + \| p \|^2_{\grad}  \right)^{1/2}   \left( \| \vec{C} \|^2_{\dive} + \| \vec{F} \|^2_{\curl} + \| q \|^2_{\grad}  \right)^{1/2}.
\end{equation}
Thus, the auxiliary problem, \eqref{aux-mixed-form-B}-\eqref{aux-mixed-form-p}, is well-posed. 
\end{theorem}

\begin{proof}
Choose $\vec{C} = \vec{B} + \frac{\tau}{2} \, \curl \, \vec{E}$,
$\vec{F} = \vec{E} + \frac{\tau}{2} \, \grad \, p $, and $q = p$.  Then, 
\begin{align*}
a(\vec{B}, \vec{E}, p; \vec{C}, \vec{F}, q ) & =  \frac{ 2}{\tau}  \scalar{ \mu^{-1} \vec{B} }{\vec{B} + \frac{\tau}{2} \, \curl \, \vec{E}} + \scalar{ \mu^{-1}\curl\, \vec{E}}{\vec{B} + \frac{\tau}{2} \, \curl \, \vec{E}} + \scalar{ \dive\, \vec{B}}{\dive \, \vec{B}} \\
& \quad + \frac{2}{\tau} \scalar{ \varepsilon \vec{E}}{ \vec{E} + \frac{\tau}{2} \, \grad \, p} + \scalar{ \varepsilon \,  \grad \, p}{ \vec{E} + \frac{\tau}{2} \, \grad \, p} - \scalar{ \mu^{-1} \vec{B} }{\curl\,  \vec{E}} \\ 
& \quad + (1+\gamma)\scalar{  \vec{E}}{ \vec{E}}_{\Gamma_i}  + \frac{2}{\tau}\scalar{p}{p} - \scalar{\varepsilon \vec{E} }{\grad \, p},
\end{align*}
where we use the facts that $\dive \, \curl \, \vec{E} = 0$, $\curl \,
\grad \, p = 0$, and $\int_{\Gamma_i}
\scalar{\vec{E}_{\textrm{tan}}}{\grad \, p} = 0$.  Then, after some rearranging,
\begin{align*}
a(\vec{B}, \vec{E}, p; \vec{C}, \vec{F}, q ) & = \frac{2}{\tau} \| \vec{B} \|^2_{\mu^{-1}} + \scalar{ \mu^{-1}\vec{B}}{\curl \, \vec{E}} + \| \dive \, \vec{B} \|^2\\
& \quad + \frac{2}{\tau} \| \vec{E} \|_{\varepsilon}^2 + \frac{\tau}{2} \| \curl \, \vec{E} \|^2_{\mu^{-1}} + (1+\gamma) \| \vec{E} \|^2_{\Gamma_i} + \scalar{\varepsilon \vec{E}}{\grad \, p} \\
& \quad + \frac{2}{\tau} \| p \|^2 + \frac{\tau}{2} \| \grad \, p\|^2_{\varepsilon} \\
& \geq \frac{2}{\tau} \| \vec{B} \|^2_{\mu^{-1}} - \frac{1}{\tau}  \| \vec{B} \|^2_{\mu^{-1}} - \frac{\tau}{4} \| \curl \, \vec{E} \|^2_{\mu^{-1}}  + \| \dive \, \vec{B} \|^2 \\
& \quad + \frac{2}{\tau} \| \vec{E} \|_{\varepsilon}^2 + \frac{\tau}{2} \| \curl \, \vec{E} \|^2_{\mu^{-1}} + (1+\gamma) \| \vec{E} \|^2_{\Gamma_i} - \frac{1}{\tau}  \| \vec{E} \|^2_{\varepsilon}  - \frac{\tau}{4}  \| \grad \, p\|^2_{\varepsilon} \\
& \quad + \frac{2}{\tau} \| p \|^2 + \frac{\tau}{2}\| \grad \, p\|^2_{\varepsilon} \\
& = \frac{1}{\tau} \| \vec{B} \|^2_{\mu^{-1}} + \| \dive \, \vec{B} \|^2  + \frac{1}{\tau} \| \vec{E} \|_{\varepsilon}^2 +  \frac{\tau}{4} \| \curl \, \vec{E} \|^2_{\mu^{-1}} + (1+\gamma) \| \vec{E} \|^2_{\Gamma_i} \\
& \quad +  \frac{2}{\tau} \| p \|^2 + \frac{\tau}{4}\| \grad \, p\|^2_{\varepsilon} \\
& \geq \frac{1}{2} \left(  \| \vec{B} \|^2_{\dive} + \| \vec{E} \|^2_{\curl} + \| p \|^2_{\grad}  \right).
\end{align*}
On the other hand, 
\begin{align*}
\| \vec{C} \|^2_{\dive} + \| \vec{F} \|^2_{\curl} + \| q \|^2_{\grad} &= \| \vec{B} + \frac{\tau}{2} \, \curl \, \vec{E} \|^2_{\dive} + \| \vec{E} + \frac{\tau}{2} \, \grad \, p   \|^2_{\curl} + \| p \|^2_{\grad} \\
& \leq  2 \| \vec{B} \|^2_{\dive} + \frac{\tau^2}{2} \| \curl \, \vec{E} \|^2_{\dive} + 2 \| \vec{E} \|^2_{\curl} + \frac{\tau^2}{2} \| \grad \, p \|^2_{\curl} + \| p \|^2_{\grad} \\
&= 2 \| \vec{B} \|^2_{\dive} + \tau  \| \curl \, \vec{E} \|^2_{\mu^{-1}} + 2 \| \vec{E} \|^2_{\curl} + \tau \| \grad \, p \| ^2_{\varepsilon} +  \| p \|^2_{\grad}  \\
& \leq 4 \left(  \| \vec{B} \|^2_{\dive} + \| \vec{E} \|^2_{\curl} + \| p \|^2_{\grad}  \right).
\end{align*}
Then, the inf-sup condition, \eqref{ine:inf-sup}, follows directly.
Boundedness, \eqref{ine:bounded}, is derived from the definition of
the bilinear form, $a(\cdot, \cdot, \cdot; \cdot, \cdot, \cdot)$, and
some Cauchy-Schwarz inequalities.  Finally, the well-posedness of the auxiliary problem, \eqref{aux-mixed-form-B}-\eqref{aux-mixed-form-p}, follows by applying the Babuska-Brezzi theory. 
\end{proof}

\begin{theorem} \label{thm:well-posed}
If $\vec{g}_{\vec{B}} \in \left ( H^0_{h,\textrm{imp}}(\dive)\right )'$, the mixed formulation, \eqref{discrete-mixed-form-B}-\eqref{discrete-mixed-form-p}, is well-posed.
\end{theorem}

\begin{proof}
Since \eqref{discrete-mixed-form-B}-\eqref{discrete-mixed-form-p} and
\eqref{aux-mixed-form-B}-\eqref{aux-mixed-form-p} are equivalent, and
the latter is well-posed, then so is the original mixed formulation,
\eqref{discrete-mixed-form-B}-\eqref{discrete-mixed-form-p}.  Similar arguments as in Lemma 1 and Theorem 8 of
\cite{Hu.K;Ma.Y;Xu.J.2014a} give the result.
\end{proof}

\section{Robust Linear Solvers} \label{sec:solver}
Next, we develop the robust linear solvers for solving
\eqref{discrete-mixed-form-B}-\eqref{discrete-mixed-form-p}.  We 
consider two types of preconditioners.  One is based on the
well-posedness described above, and the other is based on block factorization.   

\subsection{Block Preconditioners based on Well-posedness}
The first type of preconditioner we consider follows from the
framework proposed in \cite{Loghin.D;Wathen.A2004} and
\cite{Mardal.K;Winther.R2010}.  Such preconditioners are constructed
based on the well-posdeness of the linear system.  Roughly speaking,
the well-posedness shows that the linear operator under consideration
is an isomorphism from the given Hilbert space to its dual.
Therefore, any isomorphism from the dual space back to the original
Hilbert space can be used as a preconditioner.  A natural choice for
such an isomorphism is the Riesz operator induced by the norm equipped
by the Hilbert space.  

\subsubsection{Preconditioner for the Auxiliary Problem}
First consider the auxiliary problem used in the proof of well-posedness. The matrix form of \eqref{aux-mixed-form-B}-\eqref{aux-mixed-form-p} is 
\begin{equation}\label{def:aux-operator}
\bigAaux \vec{x} = \vec{b} \Longleftrightarrow 
\begin{pmatrix}
\frac{2}{\tau} \massM_{\bm{B}} +  \matD^T \massM_0 \matD & \massM_{\bm{B}} \matK &   \\
- \matK^T \massM_{\bm{B}} & \frac{2}{\tau} \massM_{\bm{E}} + \matZ & \massM_{\bm{E}} \matG \\
   	& - \matG^T \massM_{\bm{E}} & \frac{2}{\tau} \massM_p  
\end{pmatrix}  
\begin{pmatrix}
\vec{B} \\
\vec{E} \\
p
\end{pmatrix}
= 
\begin{pmatrix}
\vec{g}_{\vec{B}}\\
\vec{g}_{\vec{E}} \\
g_{p}
\end{pmatrix},
\end{equation}
where $\massM_p$, $\massM_{\bm{E}}$, $\massM_{\bm{B}}$, and $\massM_0$ are the (weighted)
mass matrices for finite-element spaces $H_{h,0}(\grad)$,
$H_{h,\textrm{imp}}(\curl)$, $H_{h,\textrm{imp}}(\dive)$, and $L_h^2$,
respectively, and $\matZ$ represents the surface integral associated with the
impedance boundary condition.  Additionally, $\matG$, $\matK$, and $\matD$ are
incidence matrices representing the discrete gradient, curl, and
divergence operators on the given triangulation.  Let $\{\phi_i^{\grad} \}$,
$\{ \bm{\phi}_i^{\curl} \}$, and $\{ \bm{\phi}_i^{\dive}  \}$ be the
basis of $H_{h,0}(\grad)$, $H_{h,\textrm{imp}}(\curl)$, and
$H_{h,\textrm{imp}}(\dive)$, respectively.  Moreover, let $\{
\bm{\eta}_i^{\curl} \}$, $\{ \bm{\eta}_i^{\dive} \}$, and $\{
\eta_i^{L^2}  \}$ be the corresponding degrees of freedom.  Then, $\matG$, $\matK$, and $\matD$ are defined as follows:
\begin{align*}
\matG_{ij} &:= \bm{\eta}^{\curl}_i(\grad \; \phi_j^{\grad} ) \\
\matK_{ij}  &:= \bm{\eta}^{\dive}_i(\curl \; \bm{\phi}_j^{\curl} ) \\
\matD_{ij} &:= \eta^{L^2}_i(\dive \; \bm{\phi}_j^{\dive} )
\end{align*}
Based on this definition, we naturally have 
\begin{equation*}
\matK\matG = \bm{0} \quad \text{and} \quad \matD\matK = \bm{0},
\end{equation*}
which are the discrete counterparts of $\curl \; \grad = 0$ and $\dive
\; \curl = 0$.  Another crucial property on the discrete level is $\matG^T \matZ = 0$.  This follows from the fact that
\begin{equation*}
\scalar{\matZ \vec{E}}{\grad \, p} = (1+ \gamma) \int_{\Gamma_i} \scalar{\vec{n}\wedge\vec{E}}{ \vec{n} \wedge \grad \, p} = 0, \quad \forall \vec{E} \in H_{h, \textrm{imp}}(\curl), \ p \in H_{h,0}(\grad).
\end{equation*}
Note that these properties hold for any order of
finite-element spaces as long as the spaces satisfy the exact sequence
in \eqref{def:exact}.


Based on this framework, we first consider the following block diagonal preconditioner, which corresponds to the Reisz operator induced by the weighted norm $\| \cdot \|_{\dive}$, $\| \cdot \|_{\curl}$, and $\| \cdot \|_{\grad}$:
\begin{equation*} 
\widetilde{\preconWaux_{\prediag}} =
\begin{pmatrix}
\matD^T \massM_0 \matD  + \frac{2}{\tau} \massM_{\bm{B}} &  0 &  0 \\
0  & \frac{\tau}{2} \matK^T \massM_{\bm{B}} \matK+ \frac{2}{\tau} \massM_{\bm{E}}  + \matZ & 0 \\
0 &  0  & \frac{\tau}{2} \matG^T \massM_p \matG+  \frac{2}{\tau} \massM_p 
\end{pmatrix}^{-1}.
\end{equation*}
Together with the well-posedness of the auxiliary problem (Theorem
\ref{thm:well-posed-aux}) and the results in
\cite{Loghin.D;Wathen.A2004,Mardal.K;Winther.R2010}, the condition
number of the preconditioned system, $\kappa(\widetilde{\preconWaux_{\prediag}}\bigAaux) = O(1)$, which implies that $\widetilde{\preconWaux_{\prediag}}$ is a robust preconditioner. 

In practice, the action of $\widetilde{\preconWaux_{\prediag}}$ involves
the inversion of three diagonal blocks, which could be expensive.   In order to reduce the cost, we replace the diagonal blocks of $\widetilde{\preconWaux_{\prediag}}$ by their spectral equivalent symmetric positive definite (SPD) approximations:
\begin{equation*}
\preconWaux_{\prediag} = \textrm{diag}\left(  \preQ_{\bm{B}}, \preQ_{\bm{E}}, \preQ_{p} \right),
\end{equation*}
Using HX-preconditioners \cite{Hiptmair.R;Xu.J2007} for
$\preQ_{\bm{B}}$ and $\preQ_{\bm{E}}$ and standard multigrid (MG)
preconditioners for $\preQ_{p}$, it is shown that the
condition number $\kappa(\preconWaux_{\prediag}\bigAaux) = O(1)$ \cite{Mardal.K;Winther.R2010}. 

\subsubsection{Preconditioner for the Original Formulation}
Next, we consider the original structure-preserving discretization,
\eqref{discrete-mixed-form-B}-\eqref{discrete-mixed-form-p}.  In
matrix form, we write,
\begin{equation}\label{def:discrete-operator}
\bigA \vec{x} = \vec{b} \Longleftrightarrow 
\begin{pmatrix}
\frac{2}{\tau} \massM_{\bm{B}} & \massM_{\bm{B}} \matK &   \\
- \matK^T \massM_{\bm{B}} & \frac{2}{\tau} \massM_{\bm{E}} + \matZ & \massM_{\bm{E}} \matG \\
   	& - \matG^T \massM_{\bm{E}} & \frac{2}{\tau} \massM_p  
\end{pmatrix}  
\begin{pmatrix}
\vec{B} \\
\vec{E} \\
p
\end{pmatrix}
= 
\begin{pmatrix}
\vec{g}_{\vec{B}}\\
\vec{g}_{\vec{E}} \\
g_{p}
\end{pmatrix},
\end{equation}
which is obtained by removing the stabilization term, $\matD^T
\massM_0 \matD$, in $\bigAaux$.  Removing the stabilization term in
the preconditioner $\widetilde{\preconWaux}_{\prediag}$, then, we obtain a diagonal block preconditioner for $\bigA$:
\begin{equation}\label{def:B_D} 
\widetilde{\preconW}_{\prediag} =
\begin{pmatrix}
 \frac{2}{\tau} \massM_{\bm{B}} &  0 &  0 \\
0  & \frac{\tau}{2} \matK^T \massM_{\bm{B}} \matK+ \frac{2}{\tau} \massM_{\bm{E}}  + \matZ & 0 \\
0 &  0  & \frac{\tau}{2} \matG^T \massM_p \matG+  \frac{2}{\tau} \massM_p 
\end{pmatrix}^{-1}.
\end{equation}
Using the fact that $\matD \matK = \bm{0}$, we have,
\begin{equation*}
\left( \matD^T \massM \matD  + \frac{2}{\tau} \massM_{\bm{B}} \right)^{-1}   \massM_{\bm{B}} \matK = \frac{\tau}{2} \matK = \left( \frac{2}{\tau} \massM_{\bm{B}} \right)^{-1}   \massM_{\bm{B}} \matK.
\end{equation*}
Therefore, $\widetilde{\preconWaux}_{\prediag} \bigAaux =
\widetilde{\preconW}_{\prediag} \bigA$, which implies that
$\kappa(\widetilde{\preconW}_{\prediag} \bigA) = O(1)$ and
$\widetilde{\preconW}_{\prediag}$ is a robust preconditioner for
$\bigA$.  Obviously, the action of $\widetilde{\preconW}_{\prediag}$
can be expensive in practice, so we replace the diagonal blocks of $\widetilde{\preconW}_{\prediag}$ by their spectral equivalent SPD approximations:
\begin{equation}\label{def:M_D}
\preconW_{\prediag} = \textrm{diag}\left(  \preQ_{\bm{B}}, \preQ_{\bm{E}}, \preQ_{p} \right).
\end{equation}
It is easy to see that $\kappa(\preconW_{\prediag} \bigA) = O(1)$ and $\preconW_{\prediag}$ is a robust preconditioner for $\bigA$.

\subsubsection{Keeping the Magnetic Field Solenoidal}
In \cite{2013AdlerJ_PetkovV_ZikatanovL-aa}, we show that an important
feature of the structure-preserving discretization, \eqref{discrete-mixed-form-B}-\eqref{discrete-mixed-form-p}, is that it keeps
$\dive\,\vec{B} = 0$ at every time step.  Here, we follow the approach
proposed in \cite{Hu.K;Hu.X;Xu.J;Ma.Y.2014a} to show that it is
possible to preserve the divergence-free condition for each iteration
of the linear solver. 

\begin{theorem}\label{thm:B-D-div-free}
Assume the initial guess, $\vec{x}^0 = (\vec{B}^0, \vec{E}^0, p^0)^T$,
and right-hand side, $\vec{b} = (\vec{g}_{\vec{B}}, \vec{g}_E, g_p)^T$,
satisfy $\dive \, \vec{B}^0 = 0$ and $\dive \, \massM^{-1}_{\bm{B}}
\vec{g}_{\vec{B}} = 0$, respectively.  Then, all iterations, $\vec{x}^l
= (\vec{B}^l, \vec{E}^l, p^l)^T$, of the
$\widetilde{\preconW}_{\prediag}$ preconditioned GMRES method satisfy $\dive \, \vec{B}^l = 0$.
\end{theorem}
\begin{proof}
According to the definition of preconditioned GMRES, we have
\begin{equation*}
\vec{x}^l \in \vec{x}^0 + \mathcal{K}^l(\widetilde{\preconW}_{\prediag} \bigA, \vec{r}^0), 
\end{equation*}
where, 
$$\mathcal{K}^l(\widetilde{\preconW}_{\prediag} \bigA, \vec{r}^0)
= \textrm{span}\{  \vec{r}^0, \widetilde{\preconW}_{\prediag} \bigA
\vec{r}^0,  \left( \widetilde{\preconW}_{\prediag} \bigA \right)^2
\vec{r}^0, \cdots,  \left( \widetilde{\preconW}_{\prediag} \bigA \right)^{l-1}\vec{r}^0 \},$$
and $\vec{r}^0 = (\vec{r}_{\vec{B}}^0, \vec{r}^0_{\vec{E}},
\vec{r}^0_p)^T := \widetilde{\preconW}_{\prediag} (\vec{b} -  \bigA
\vec{x}^0) $.  Note that $\dive \, \vec{r}^0_{\vec{B}} = 0$.

Denote $\vec{v}^m = (\vec{v}^m_{\vec{B}}, \vec{v}^m_{\vec{E}},
\vec{v}^m_p)^T := \left( \widetilde{\preconW}_{\prediag} \bigA
\right)^m \vec{r}^0$, $m = 0,1,2, \cdots, l-1$.  Since $\vec{v}^m =
\widetilde{\preconW}_{\prediag} \bigA  \vec{v}^{m-1}$, we obtain,
\begin{equation} \label{eqn:B-relation}
\vec{v}^m_{\vec{B}} = \left(  \frac{\tau}{2} \massM_{\bm{B}} \right)^{-1} \left(  \frac{\tau}{2} \massM_{\bm{B}} \vec{v}^{m-1}_{\vec{B}} + \massM_{\bm{B}} \matK  \vec{v}^{m-1}_{\vec{B}} \right) = \vec{v}^{m-1}_{\vec{B}} + \frac{2}{\tau} \matK \vec{v}^{m-1}_{\vec{B}}.
\end{equation}
Then, $\dive \, \vec{v}^m_{\vec{B}} = 0$ if $\dive \,
\vec{v}^{m-1}_{\vec{B}} = 0$.  Since $\dive \, \vec{r}^0_{\vec{B}} =
0$, by induction, we have\\
 $\dive \, \vec{v}^m_{\vec{B}} = 0$.  

Finally, $\vec{x}^l$ is a linear combination of $\vec{v}^m$, $m=0,1,2\cdots, l-1$, which implies that $\vec{B}^l$ is a linear combination of $\vec{v}^m_{\vec{B}}$. Since $\dive \, \vec{v}^m_{\vec{B}} = 0$, we conclude that $\dive \, \vec{B}^l  = 0$ for all $l$.
\end{proof}

The above theory says that using $\widetilde{\preconW}_{\prediag}$ as
a preconditioner preserves the divergence-free condition of $\vec{B}$.  However, the preconditioner $\preconW_{\prediag}$, in general, may not.  A remedy is to use $\preQ_{\bm{B}} = \left(  \frac{\tau}{2} \massM_{\bm{B}} \right)^{-1}$, which leads to
\begin{equation}\label{def:M_D-div-free}
\preconW_{\prediag} = \textrm{diag}\left( \left( \frac{\tau}{2} \massM_{\bm{B}} \right)^{-1}, \preQ_{\bm{E}}, \preQ_{p} \right).
\end{equation}
While it may seem impractical to use such a preconditioner, because of
the need to invert the mass matrix exactly, using
\eqref{eqn:B-relation}, we can update $\vec{v}_B^{m}$ without this
inversion.  Thus, using $\preconW_{\prediag}$ as the
preconditioner still allows for the preservation of the divergence-free condition for all the iterations of GMRES.

\begin{theorem}\label{thm:M_D-div-free}
Assume the initial guess, $\vec{x}^0 = (\vec{B}^0, \vec{E}^0, p^0)^T$,
and right-hand side, $\vec{b} = (\vec{g}_{\vec{B}}, \vec{g}_E, g_p)^T$,
satisfy $\dive \, \vec{B}^0 = 0$ and $\dive \, \massM^{-1}_{\bm{B}}
\vec{g}_{\vec{B}} = 0$, respectively.  Then, all iterations, $\vec{x}^l
= (\vec{B}^l, \vec{E}^l, p^l)^T$, of the
$\preconW_{\prediag}$ preconditioned GMRES method satisfy $\dive \, \vec{B}^l = 0$.
\end{theorem}

\begin{proof}
The proof is the same as for Theorem \ref{thm:B-D-div-free} with $\widetilde{\preconW}_{\prediag}$ replaced by $\preconW_{\prediag}$.
\end{proof}

\subsubsection{Generalization} We conclude this subsection with the
generalization of the block diagonal
preconditioner to a block triangular preconditioner, 
\begin{equation}\label{def:M_L-div-free}
\preconW_{\prelow} = 
\begin{pmatrix}
\left( \frac{\tau}{2} \massM_{\vec{B}}  \right)^{-1} & 0 & 0 \\
- \matK^T\massM_{\vec{B}}  &  \preQ_{\vec{E}}^{-1} &  0 \\
 0   &  -\matG^T \massM_{\vec{E}} & \preQ_p^{-1}
\end{pmatrix}^{-1},
\end{equation}
and 
\begin{equation}\label{def:M_U-div-free}
\preconW_{\preupp} = 
\begin{pmatrix}
\left( \frac{\tau}{2} \massM_{\vec{B}}  \right)^{-1} & \massM_{\vec{B}} \matK & 0 \\
0 &  \preQ_{\bm{E}}^{-1} & \massM_{\vec{E}} \matG  \\
 0   &  0 & \preQ_p^{-1}
\end{pmatrix}^{-1}.
\end{equation}
Since the analysis for $\preconW_{\preupp}$ is the same, we only
consider $\preconW_{\prelow}$ here.   Also, note that we use
$\frac{\tau}{2} \massM_{\bm{B}}$ for the first diagonal block in order
to keep the divergence-free condition.  

With a slight abuse of notation, we define $A_{\vec{B}}$, $A_{\vec{E}}$, and $A_p$ as follows:
\begin{align*}
\scalar{A_{\vec{B}} \vec{B}}{\vec{C}} & = \scalar{\vec{B}}{\vec{C}}_{\dive}, \ \forall \vec{C} \in H_{h, \text{imp}}(\dive) \\
\scalar{A_{\vec{E}} \vec{E}}{\vec{F}} & = \scalar{\vec{E}}{\vec{F}}_{\curl}, \ \forall \vec{F} \in H_{h, \text{imp}}(\curl) \\
\scalar{A_{p} p}{q} & = \scalar{p}{q}_{\grad}, \ \forall q \in H_{h,0}(\grad).
\end{align*}
Note that $\preQ_{\vec{E}}$ and $\preQ_p$ are spectrally equivalent to
the inverse of of $A_{\vec{E}}$ and $A_{p}$:
\begin{align}
& c_{1, \vec{E}} \scalar{\preQ_{\vec{E}} \, \vec{E}}{\vec{E}} \leq \scalar{ A_{\vec{E}}^{-1}\, \vec{E}}{\vec{E}} \leq c_{2, \vec{E}} \scalar{\preQ_{\vec{E}} \, \vec{E}}{\vec{E}}, \label{ine:spectral-Q_E} \\
& c_{1, p} \scalar{\preQ_{p} \, p}{p} \leq \scalar{ A_p^{-1}\, p}{p} \leq c_{2, p} \scalar{\preQ_{p} \, p}{p}. \label{ine:spectral-Q_p}
\end{align}

Following the standard convergence analysis of GMRES, we derive the following theorem concerning the so-called \emph{Field-of-Value} of $\preconW_{\prelow}\bigA$.  Here, we use the norm $\| \cdot \|_{\preconW^{-1}}$ induced by $\preconW = \text{diag} \left( A_{\vec{B}}^{-1}, \preQ_{\vec{E}}, \preQ_p \right)$.

\begin{theorem}\label{thm:wp}
Assume \eqref{ine:spectral-Q_E} and \eqref{ine:spectral-Q_p} hold, then there exists constants, $\lambda$ and $\Lambda$, such that for any $\vec{x} \neq \vec{0}$, 
\begin{equation*}
\lambda \leq \frac{ \scalar{\preconW_{\prelow} \bigA \vec{x}}{\vec{x}}_{\preconW^{-1}} }{\scalar{\vec{x}}{\vec{x}}_{\preconW^{-1}}}, \quad \frac{\| \preconW_{\prelow}^{-1} \bigA \vec{x} \|_{\preconW_{\prelow}^{-1}}}{\| \vec{x} \|_{\preconW_{\prelow}^{-1}}} \geq \Lambda,
\end{equation*}
provided $ \| I_{\vec{E}} - \preQ_{\vec{E}} A_{\vec{E}} \|_{A_{\vec{E}}} \leq \rho < \sqrt{3}-1$.  Here, the constants $\lambda$ and $\Lambda$ do not depend on neither the discretization parameters, $h$ and $\tau$, nor the physical parameters, $\varepsilon$ and $\mu^{-1}$.  
\end{theorem}

\begin{proof}
By the definition of $\preconW_{\prelow}^{-1}$ and $\bigA$, we have
\begin{align*}
\scalar{\preconW_{\prelow}\bigA\vec{x}}{\vec{x}}_{\mathcal{W}^{-1}} & = \scalar{\vec{B}}{\vec{B}}_{A_{\vec{B}}} + \scalar{\frac{\tau}{2} \curl \; \vec{E}}{\vec{B}}_{A_{\vec{B}}} + \scalar{\vec{E}}{\vec{E}}_{A_{\vec{E}}} + \scalar{\vec{E}}{\grad \; p} \\
& \quad + \scalar{\preQ_{\vec{E}}A_{\vec{E}} \vec{E}}{\grad \; p} -  \scalar{\vec{E}}{\grad \; p} + \scalar{\preQ_p \grad \; p }{\grad \; p} + \scalar{\frac{\tau}{2} p}{p} \\
& \geq \| \vec{B} \|_{\dive}^2 - \| \vec{B} \|_{\dive} \sqrt{\frac{\tau}{2}} \| \curl \; \vec{E} \|_{\mu^{-1}} + \| \vec{E} \|_{\curl}^2 \\
& \quad - (1+\rho) \| \vec{E} \|_{\curl} \| \grad \; p \|_{\preQ_{\vec{E}}} + \| \grad \; p \|_{\preQ_{\vec{E}}}^2 + \frac{\tau}{2} \| p \|^2 \\
& \geq \| \vec{B} \|_{\dive}^2 - \| \vec{B} \|_{\dive} \| \vec{E} \|_{\curl} + \| \vec{E} \|_{\curl}^2 \\
& \quad - (1+\rho) \| \vec{E} \|_{\curl} \| \grad \; p \|_{\preQ_{\vec{E}}} + \| \grad \; p \|_{\preQ_{\vec{E}}}^2 + \frac{\tau}{2} \| p \|^2  \\
& \geq 
\begin{pmatrix}
\| \vec{B} \|_{\dive} \\
\| \vec{E} \|_{\curl} \\
\| \grad \; p \|_{\preQ_{\vec{E}}} \\
\sqrt{\frac{\tau}{2}}\| p \|
\end{pmatrix}^T 
\begin{pmatrix}
1 & -\frac{1}{2} & 0  &  0\\
-\frac{1}{2} &  1  &  -\frac{1+\rho}{2} & 0 \\
0 & - \frac{1+\rho}{2} & 1 & 0 \\
0 &  0  &  0  & 1
\end{pmatrix}
\begin{pmatrix}
\| \vec{B} \|_{\dive} \\
\| \vec{E} \|_{\curl} \\
\| \grad \; p \|_{\preQ_{\vec{E}}} \\
\sqrt{\frac{\tau}{2}}\| p \|
\end{pmatrix}.
\end{align*} 
It is easy verify that the matrix in the middle is SPD, when $0 \leq \rho < \sqrt{3}-1$.  Therefore, there exists a constant $\lambda_0$ such that,
\begin{align*}
\scalar{\preconW_{\prelow}\bigA\vec{x}}{\vec{x}}_{\mathcal{W}^{-1}} & \geq \lambda_0 \left( \| \vec{B} \|_{\dive}^2 + \| \vec{E} \|_{\curl} ^2 +  \| \grad \; p \|_{\preQ_{\vec{E}}}^2 +  \frac{\tau}{2} \| p \|^2 \right) \\
& \geq \lambda_0 \left( \| \vec{B} \|_{\dive}^2 + \| \vec{E} \|_{\curl} ^2 +  c_{2,\vec{E}}^{-1} \; \frac{2}{\tau} \| \grad \; p \|_{\varepsilon}^2 +  \frac{\tau}{2} \| p \|^2 \right) \\
& \geq \min \{ 1, (1-\rho), c_{2,\vec{E}}^{-1} c_{1, p}^{-1}, c_{1,p}^{-1}  \} \lambda_0 \scalar{\vec{x}}{\vec{x}}_{\mathcal{W}^{-1}},
\end{align*}
which gives the lower bound $\lambda :=  \min \{ 1, (1-\rho),
c_{2,\vec{E}}^{-1} c_{1, p}^{-1}, c_{1,p}^{-1}  \} \lambda_0$.  The upper bound, $\Lambda$, follows directly from the continuity of each term.
\end{proof}

The condition $ \| I_{\vec{E}} - \preQ_{\vec{E}} A_{\vec{E}}
\|_{A_{\vec{E}}} \leq \rho < \sqrt{3}-1$ means that we should solve
$A_{\vec{E}}$ to a certain accuracy in practice.  Regardless, the above
theorem implies that $\preconW_{\prelow}$ preconditioned GMRES converges uniformly with respect to the discretization and physical parameters.

\subsection{Block Preconditioner based on Exact Block Factorization}
Next, we consider linear solvers based on block factorization.  In
general, block factorization inevitably
involves systems with Schur complements, often built recursively if
the system involves more than two fields. Since exact Schur
complements are typically dense, traditional preconditioners based
on block factorization need approximations, and the performance of the
preconditioner strongly depends on the accuracy of these
approximations.   However, good approximations of the Schur complements are, in general, rather
challenging to design in practice.  In the case of
\eqref{discrete-mixed-form-B}-\eqref{discrete-mixed-form-p} , though, the
structure-preserving discretization allows for the Schur
complements to be computed exactly.  Specifically, the exactness property
of the sequence of discrete spaces yield sparse Schur complements that
are used directly without approximation. 

\subsubsection{Exact Block Facorization}
First, consider the mixed formulation,
\eqref{discrete-mixed-form-B}-\eqref{discrete-mixed-form-p}, more
precisely, its matrix form, \eqref{def:discrete-operator}.  Recall
that due to the
structure-preserving discretization, properties of the gradient and
curl operators (e.g., $\curl \ \grad = 0$) are carried over to the
discrete level (e.g., $\matK\matG= \vec{0}$ or, equivalently, $\matG^T
\matK^T = \vec{0}$).  Likewise $G^TZ = 0$.  Based on this, we have the following exact block factorization of \eqref{def:discrete-operator},
\begin{equation}  \label{eqn:LDU}
\bigA = \prelow \, \prediag \, \preupp,
\end{equation}
where
\begin{equation} \label{def:LDU}
\prelow = 
\begin{pmatrix}
I &  &  \\
- \frac{\tau}{2} \matK^T & I & \\
&	-\frac{\tau}{2} \matG^T & I
\end{pmatrix}, \
\prediag = 
\begin{pmatrix}
\frac{2}{\tau} \massM_{\bm{B}} &  &  \\
&	\schurS_{\vec{E}} & \\
&   &  \schurS_p
\end{pmatrix}, \
\preupp =
\begin{pmatrix}
I & \frac{\tau}{2} \matK & \\
&   I & \frac{\tau}{2} \matG \\
&	& I
\end{pmatrix},
\end{equation}
with the following Schur complements
\begin{align*}
\schurS_{\vec{E}} & =  \frac{\tau}{2} \matK^T \massM_{\bm{B}} \matK+ \frac{2}{\tau} \massM_{\bm{E}}  + \matZ, \\
\schurS_p & = \frac{\tau}{2} \matG^T \massM_p \matG+  \frac{2}{\tau} \massM_p.
\end{align*}
Again, we emphasize that, due to the structure-preserving
discretization, the Schur complements are computed exactly and are sparse. 

\subsubsection{Block Preconditioners}
Based on the above exact factorization, \eqref{eqn:LDU}, we design
several block preconditioners.  One simple choice is to use the
diagonal block, $\prediag^{-1}$.  Interestingly, such choice actually
leads to the preconditioner, $\widetilde{\preconW_{\prediag}}$,
\eqref{def:B_D}, derived from the well-posedness.  Of course,
computing the inverse of $\prediag$ involves inverting the Schur
complements, $\schurS_{\vec{E}}^{-1}$ and $\schurS_p^{-1}$, exactly, which is
expensive and infeasible in practice.  Therefore, we replace the Schur
complements by their spectral equivalent SPD approximations, which in
the diagonal case,
yields the block preconditioners in \eqref{def:M_D} (or \eqref{def:M_D-div-free} if we need to preserve the divergence-free property):  

\begin{align}
& c_{1, \vec{B}} \scalar{\preQ_{\vec{B}} \, \vec{B}}{\vec{B}} \leq \scalar{ \left(\frac{2}{\tau}\massM_{\bm{B}} \right)^{-1}\, \vec{B}}{\vec{B}} \leq c_{2, \vec{B}} \scalar{\preQ_{\vec{B}} \, \vec{B}}{\vec{B}}, \label{ine:Q_f} \\
& c_{1, \vec{E}} \scalar{\preQ_{\vec{E}} \, \vec{E}}{\vec{E}} \leq \scalar{ \schurS_{\vec{E}}^{-1}\, \vec{E}}{\vec{E}} \leq c_{2, \vec{E}} \scalar{\preQ_{\vec{E}} \, \vec{E}}{\vec{E}}, \label{ine:Q_e} \\
& c_{1, p} \scalar{\preQ_{p} \, p}{p} \leq \scalar{ \schurS_p^{-1}\, p}{p} \leq c_{2, p} \scalar{\preQ_{p} \, p}{p}. \label{ine:Q_v}
\end{align}
This implies that for $\preQ = \textrm{diag}\left(  \preQ_{\vec{B}}, \preQ_{\vec{E}}, \preQ_p \right)$, we have
\begin{equation*}
c_1 \scalar{\preQ \, \vec{x}}{\vec{x}} \leq \scalar{\prediag^{-1} \, \vec{x}}{\vec{x}} \leq c_2 \scalar{\preQ \, \vec{x}}{\vec{x}},
\end{equation*}
with $c_1 = \min \{ c_{1, \vec{B}}, c_{1,\vec{E}}, c_{1,p}  \}$ and
$c_2 = \max \{ c_{2,\vec{B}}, c_{2,\vec{E}}, c_{2,p} \}$.  Possible
choices of $\preQ_{\vec{B}}$, $\preQ_{\vec{E}}$, and $\preQ_p$ were
discussed in the previous section.  Again, we choose $\preQ_{\vec{B}} =  \left(\frac{2}{\tau}\massM_{\bm{B}} \right)^{-1}$ in order to preserve the divergence-free condition in the linear solver.  

Based on $\preQ$, though, we consider three other different block preconditioners,
\begin{equation} \label{def:block-prec}
\preconBF_{\prelow\prediag} := \preQ \prelow^{-1}, \quad \preconBF_{\prediag\preupp} := \preupp^{-1} \preQ, \quad \preconBF_{\prelow\prediag\preupp} := \preupp^{-1} \preQ \prelow^{-1}.
\end{equation}
Here, $\prelow^{-1}$ and $\preupp^{-1}$ can be computed exactly as follows
\begin{equation*}
\prelow^{-1} = 
\begin{pmatrix}
I &  &  \\
\frac{\tau}{2} \matK^T & I& \\
&	\frac{\tau}{2} \matG^T & I 
\end{pmatrix}, \quad
\preupp^{-1} = 
\begin{pmatrix}
I & -\frac{\tau}{2} \matK & \\
&   I & -\frac{\tau}{2} \matG \\
&	& I
\end{pmatrix}.
\end{equation*}

\begin{theorem} \label{thm:block-prec}
Let $\preconBF_{\prelow\prediag}$, $\preconBF_{\prediag\preupp}$, and $\preconBF_{\prelow\prediag\preupp}$ be defined by \eqref{def:block-prec} and assume the spectral-equivalent properties, \eqref{ine:Q_f}-\eqref{ine:Q_v}, hold, then,
\begin{equation}\label{def:spectral-block-prec}
\lambda( \preconBF_{\prelow\prediag} \bigA ) \in [C_1, C_2], \ \lambda( \preconBF_{\prediag\preupp} \bigA ) \in [C_1, C_2], \ \text{and} \ \lambda( \preconBF_{\prelow\prediag\preupp} \bigA ) \in [C_1, C_2],
\end{equation}
where $C_1 = \min\{ c_{2,\vec{B}}^{-1}, c_{2,\vec{E}}^{-2},
c_{2,p}^{-1} \}$ and $C_2 = \max\{ c_{1,\vec{B}}^{-1},
c_{1,\vec{E}}^{-1}, c_{1,p}^{-1} \}$ are constants that do not depend on neither the discretization parameters, $h$ and $\tau$, nor the physical parameters, $\varepsilon$ and $\mu^{-1}$.
\end{theorem}

\begin{proof}
First consider $\preconBF_{\prelow\prediag} \bigA$, 
\begin{align*}
\preconBF_{\prelow\prediag} \bigA & = \preQ \prelow^{-1} \prelow \mathcal{D} \preupp = \preQ \mathcal{D} \preupp =
\begin{pmatrix}
\preQ_{\vec{B}} \left( \frac{2}{\tau} \massM_{\bm{B}} \right) & \preQ_{\vec{B}} \massM_{\bm{B}} \matK &  \\
&  \preQ_{\vec{E}} \schurS_{\vec{E}}  & \preQ_{\vec{E}} \massM_{\bm{B}} \matG \\
&  & \preQ_p \schurS_p
\end{pmatrix}.
\end{align*}
Since $\preconBF_{\prelow\prediag} \bigA$  is block upper triangular, its eigenvalues, $\lambda(\preconBF_{\prelow\prediag} \bigA)$, are determined by the eigenvalues of its diagonal blocks.  Then, using the spectral-equivalent properties, \eqref{ine:Q_f}-\eqref{ine:Q_v}, we have $\lambda( \preconBF_{\prelow\prediag} \bigA ) \in [C_1, C_2]$.

For the eigenvalues of $\preconBF_{\prediag\preupp} \bigA$, we consider the following generalized eigenvalue problem,
\begin{align*}
\bigA \vec{x} = \lambda \preconBF_{\prediag\preupp}^{-1} \vec{x} \ \Longleftrightarrow \ \prelow\prediag\preupp \vec{x} = \lambda \preQ^{-1} \preupp \vec{x} \ \Longleftrightarrow \ \preQ \prelow \mathcal{D} \vec{y} = \lambda \vec{y}, \ \text{where} \ \vec{y} = \preupp\vec{x}.
\end{align*}
Thus, the eigenvalues of $\preconBF_{\prediag\preupp} \bigA$ are also the eigenvalues of $\preQ\prelow\prediag$,
\begin{align*}
\preQ \prelow\prediag= 
\begin{pmatrix}
\preQ_{\vec{B}} \left( \frac{2}{\tau} \massM_{\bm{B}} \right)  &    &  \\
- \preQ_{\vec{E}} \matK^T \massM_{\bm{B}} & \preQ_{\vec{E}} \schurS_{\vec{E}} &  \\
  &  - \preQ_p \matG^T \massM_{\bm{E}} & \preQ_p \schurS_p 
  \end{pmatrix}.
\end{align*}
This is a block lower triangular matrix, and the eigenvalues are again
determined by the eigenvalues of its diagonal blocks.  Therefore, using \eqref{ine:Q_f}-\eqref{ine:Q_v}, $\lambda( \preconBF_{\prediag\preupp} \bigA ) \in [C_1, C_2]$.  

Finally, we consider $\preconBF_{\prelow\prediag\preupp}$ using the following generalized eigenvalue problem,
\begin{equation*}
\bigA\vec{x}  = \lambda \preconBF_{\prelow\prediag\preupp}^{-1} \vec{x} \ \Longleftrightarrow \ \mathcal{\prelow\prediag\preupp} \vec{x} = \lambda \prelow \preQ^{-1} \preupp \vec{x}   \ \Longleftrightarrow \  \preQ \mathcal{D} \vec{y} = \lambda \vec{y}, \ \text{where} \ \vec{y} = \preupp\vec{x}.
\end{equation*}
Then, the eigenvalues of $\preconBF_{\prelow\prediag\preupp} \bigA$
are also the eigenvalues of $\preQ\mathcal{D}$.  Since
$\preQ\mathcal{D} = \textrm{diag} ( \preQ_{\vec{B}} \left(
  \frac{2}{\tau} \massM_{\bm{B}} \right),  \preQ_{\vec{E}}
\schurS_{\vec{E}},  \preQ_p \schurS_p )$, we again conclude that $\lambda( \preconBF_{\prelow\prediag\preupp} \bigA ) \in [C_1, C_2]$.  
\end{proof}

As before, using $\preQ_{\vec{B}}$ may destroy the divergence-free
property of our discretization.  Therefore, we use $\preQ_{\bm{B}} =
\left(  \frac{\tau}{2} \massM_{\bm{B}} \right)^{-1}$ to guarantee that
the resulting preconditoned GMRES approach preserves the divergence of
$\mathcal{B}$ at each iteration.

\section{Numerical Experiments} \label{sec:numerics}
Several numerical tests are done by solving system
\eqref{eq:BcE0}-\eqref{eq:divB0} using the Crank-Nicolson time
discretization and the structure-preserving space discretization
described in Section \ref{sec:fem}.  We use a test problem described
in \cite{2011ColombiniF_PetkovV_RauchJ-aa}, for which it was shown in
\cite{2013AdlerJ_PetkovV_ZikatanovL-aa} that the given discretization
accurately resolves the solution which decays exponentially in time
and space.  Here, we focus on the robustness and efficiency of the
linear solvers proposed in the previous sections.

For the computational domain, we take the area between a polyhedral
approximation of the sphere of radius $1$, and a polyhedral
approximation of a sphere of radius $4$ (see Figure \ref{fig:domain}).
The inner sphere represents the obstacle, with an impedance boundary,
and the outer sphere is considered far enough away that a Dirichlet
(perfect conductor) boundary condition is used.  In other words, we
prescribe $\vec{E}\wedge n = 0$, $\vec{B}\cdot n = 0$, and $p=0$ on
the outer sphere.  The exact solution (taken
from~\cite[Theorem~3.2]{2011ColombiniF_PetkovV_RauchJ-aa}) is given as
follows:
\begin{eqnarray}
\vec{E}_* & =& \frac{e^{r\left (|\vec{x}| + t\right )}}{|\vec{x}|^2}\left ( r^2 -
  \frac{r}{|\vec{x}|}\right ) \left (\begin{array}{c}0\\z\\-y\\\end{array}
\right ),\label{Einit}\\
\vec{B}_* &=& e^{r\left (|\vec{x}| + t\right )} \left [ \frac{1}{|\vec{x}|^3}\left ( r^2 -
  \frac{3r}{|\vec{x}|} + \frac{3}{|\vec{x}|
^2}\right ) \left (\begin{array}{c}z^2+y^2\\-xy\\-xz\\\end{array}
\right ) +
\left ( \begin{array}{c}\frac{2r}{|\vec{x}|} -
    \frac{2}{|\vec{x}|^2}\\0\\0\\\end{array} \right ) \right
],\label{Binit}\\
p_* &=& 0,
\end{eqnarray}
where
$r = 1/2\left ( 1-\sqrt{1+4/\gamma}\right )$ for various values of $\gamma$.
For the initial conditions, we use piecewise polynomial interpolants
of the exponentially-decaying solutions given in equations
\eqref{Einit}-\eqref{Binit} at $t=0$.  Further corrections of
$\vec{E}_0$ are needed to make it orthogonal to the gradients of
functions in $H_{0,h}(\grad)$ and also to the gradients of the
discrete harmonic form.  We refer to
\cite{2013AdlerJ_PetkovV_ZikatanovL-aa} for details.  Finally, for the
tests below, we take $\gamma = 0.05$ ($r = -4$). Four different mesh are used in order to test the robustness of the preconditioners with respect to the mesh size and the detailed information about the meshes can be find in Table \ref{tab:mesh}. Numerical experiments
are done using a workstation with an 8-Core 3GHz Intel Xeon
`Sandy Bridge' CPU and 256 GB of RAM.  The software used is a
finite-element and multigrid package written by the authors.

\begin{figure}[h!]
\begin{center}
\includegraphics[scale=0.25]{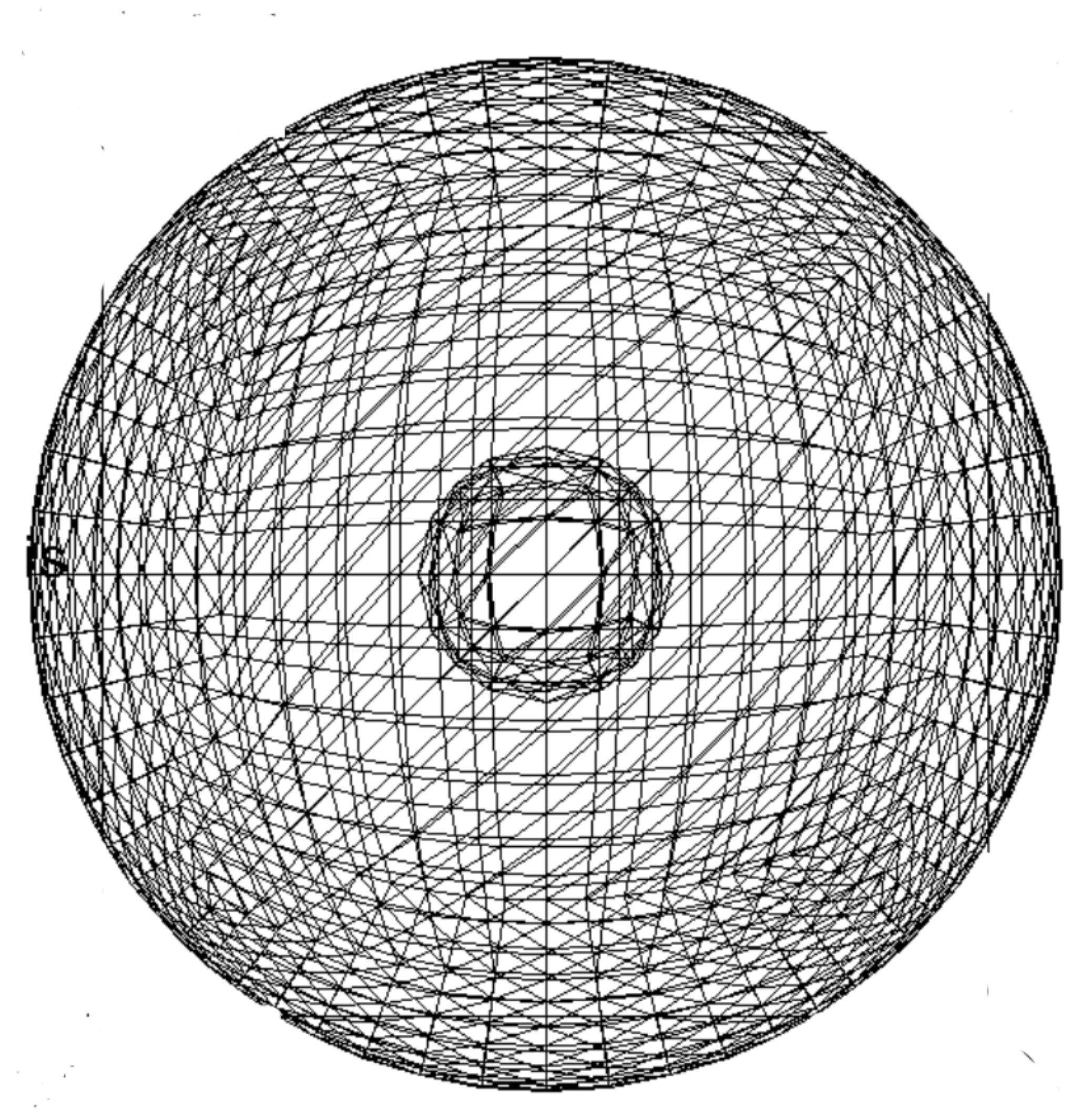}
\caption{Computational domain of the numerical tests}
\label{fig:domain}
\end{center}
\end{figure}

\begin{table}[h!]
\caption{Information of Meshes}
\begin{center}
\begin{tabular}{|c||c c c c|}
\hline 
 		&  Vertices  & Edges  	&  Faces  		& DoF \\ \hline
 Mesh 1    &  602  	  &   3,210   	&   4,812  		& 8,624  \\ 
 Mesh 2    &  3,681  	  &   21,736   	&   34,482  	& 59,899  \\ 
 Mesh 3    &  27,005    &   171,748   	&   282,962  	& 481,715   \\ 
 Mesh 4    &  228,412  &  1,525,390  	&   2,567,848  	& 4,321,650  \\ 
 \hline 
\end{tabular}
\end{center}
\label{tab:mesh}
\end{table}%

First, we consider the block preconditioners based on well-posedness:
the block diagonal preconditioner, $\preconW_{\prediag}$
\eqref{def:M_D-div-free}; the block lower triangular preconditioner,
$\preconW_\prelow$ \eqref{def:M_L-div-free}; and the block upper
triangular preconditioner, $\preconW_\preupp$
\eqref{def:M_U-div-free}.  The diagonal blocks are solved inexactly by
the preconditioned GMRES method with a tolerance of $10^{-2}$, in order to
make sure that the spectral-equivalent properties, \eqref{ine:spectral-Q_E}
and \eqref{ine:spectral-Q_p}, are satisfied.  This tolerance is
sufficient to meet the conditions in the proof of
Theorem \ref{thm:wp}.  Since the preconditioners are
actually changing at each iteration, we use flexible GMRES (FGMRES) in
the implementation with a relative residual stopping criteria of $10^{-8}$.  Table~\ref{tab:W-krylov} shows the number
iterations of the preconditioned FGMRES method with the three
different block preconditioners.  In these tests, we fix
$\varepsilon = \mu^{-1} = 1$ and investigate the robustness of the
proposed preconditioners with respect to the time step size, $\tau$, and mesh size.  The iteration counts shown in Table~\ref{tab:W-krylov} are recorded at the second time step, though
  the iterations for other time steps are similar. Based on the results, we see that the block
preconditioners are effective and robust with respect to these
parameters.

\begin{table}[h!]
\caption{Iteration counts for the block preconditioners based on
  well-posedness.  (left) Block Diagonal, $\preconW_\prediag$
  \eqref{def:M_D-div-free}. (center) Block Lower Triangular, $\preconW_\prelow$
  \eqref{def:M_L-div-free}.\\ (right) Block Upper Triangular, $\preconW_\preupp$
  \eqref{def:M_U-div-free}.  Diagonal blocks are solved inexactly.}
\begin{center}
\begin{tabular}{|l || c c c c|}
		\hline 
&\multicolumn{4}{ |c| }{$\preconW_\prediag$}\\ \hline 
		 \backslashbox{$\tau$}{Mesh} & 1 & 2 & 3 & 4
		  \\ 
		\hline 
		$0.2$ & 21 & 26 & 27 & 28  \\
		$0.1$ & 14 & 20 & 25  & 27  \\
		$0.05$ & 10 & 14 & 25  & 24 \\
		$0.025$ & 7 & 9 & 14 & 20 \\
		\hline 
		\end{tabular}
\begin{tabular}{|c c c c|}
		\hline
\multicolumn{4}{ |c| }{$\preconW_\prelow$}\\ \hline
		 1 \hspace{-65pt}\phantom{\backslashbox{$\tau$}{Mesh}}& 2 & 3  & 4
		  \\ 
		\hline 
		7 & 8 & 8 & 9 \\
		6 & 7 & 7 & 8 \\
		5 & 5 & 6 & 7  \\
		4 & 5  & 5  & 6 \\
		\hline 
		\end{tabular}
\begin{tabular}{|c c c c|}
		\hline
\multicolumn{4}{ |c| }{$\preconW_\preupp$}\\ \hline
		  1 \hspace{-65pt}\phantom{\backslashbox{$\tau$}{Mesh}} & 2 & 3  & 4
		  \\ 
		\hline 
		7 & 8 & 8 & 9  \\
		6 & 7 & 8 & 8  \\
		5 & 6 & 6 & 8 \\
		5 & 5 & 6 & 6 \\
		\hline 
		\end{tabular}
\end{center}
\label{tab:W-krylov}
\end{table}%

Next, we consider the block preconditioners based on exact block
factorization, namely, the block lower triangular preconditioner,
$\preconBF_{\prelow\prediag}$, the block upper triangular
preconditioner, $\preconBF_{\prediag\preupp}$, and the symmetric
preconditioner, $\preconBF_{\prelow\prediag\preupp}$, all defined in
\eqref{def:block-prec}.  The diagonal blocks are also solved inexactly
by preconditioned GMRES with a relative residual reduction set at
$10^{-2}$.  As before,
the outer FGMRES iterations are terminated when the value of the norm of the 
relative residual goes below 
$10^{-8}$.
Table \ref{tab:BF-krylov} shows
the number of iterations of preconditioned FGMRES with the three
different block preconditioners.  In these tests, we again fix
$\varepsilon = \mu^{-1} = 1$ and see that the block
preconditioners based on exact block factorization are effective and
robust with respect to $\tau$ and mesh size.

\begin{table}[h!]
\caption{Iteration counts for the block preconditioners based on
  block factorization.  (left) Block Lower Triangular,
  $\preconBF_{\prelow\prediag}$. (center) Block Upper Triangular,
  $\preconBF_{\prediag\preupp}$. (right) Symmetric, $\preconBF_{\prelow\prediag\preupp}$.  Diagonal blocks are solved inexactly.}
\begin{center}
\begin{tabular}{|l || c c c c|}
		\hline
&\multicolumn{4}{ |c| }{$\preconBF_{\prelow\prediag}$}\\ \hline
		 \backslashbox{$\tau$}{Mesh} & 1 & 2 & 3  & 4
		  \\ 
		\hline 
		$0.2$ & 5 & 6 & 6 & 6  \\
		$0.1$ & 5 & 5 & 6 & 5 \\
		$0.05$ & 5 & 5 & 5 & 6 \\
		$0.025$ & 4 & 5  & 5 & 5 \\
		\hline 
		\end{tabular}
\begin{tabular}{|c c c c|}
		\hline
\multicolumn{4}{ |c| }{$\preconBF_{\prediag\preupp}$}\\ \hline
		 1 \hspace{-65pt}\phantom{\backslashbox{$\tau$}{Mesh}}& 2 & 3  & 4
		  \\ 
		\hline 
		6 & 6 & 6 & 7 \\
		5 & 5 & 6 & 7 \\
		5 & 5 & 6 & 6 \\
		5 & 5 & 5 & 6  \\
		\hline 
		\end{tabular}
\begin{tabular}{|c c c c|}
		\hline
\multicolumn{4}{ |c| }{$\preconBF_{\prelow\prediag\preupp}$}\\ \hline
		  1 \hspace{-65pt}\phantom{\backslashbox{$\tau$}{Mesh}} & 2 & 3  & 4
		  \\ 
		\hline 
		4 & 4 & 4 & 5  \\
		4 & 4 & 4 & 4 \\
		4 & 4 & 4 & 4 \\
		4 & 4 & 4 & 4  \\
		\hline 
		\end{tabular}
\end{center}
\label{tab:BF-krylov}
\end{table}%

Finally, we investigate the robustness of the proposed block
preconditioners with respect to the physical parameters, $\varepsilon$
and $\mu$.  We fix the mesh size (Mesh 3 is used in all the following tests) and time step
size, $\tau = 0.1$, and consider jumps in $\varepsilon$ and $\mu$.
The tolerance of the inner GMRES iterations for solving each diagonal
block remains $10^{-2}$ for relative residual reduction and the outer FGMRES iterations are terminated
when the relative residual has norm smaller than $10^{-8}$. As before, 
the iterations count are for the second time step, with other time steps obtaining similar values. 

Table \ref{tab:jump-varepsilon} reports the number of iterations when
there is jump in $\varepsilon$, but $\mu^{-1}$ is fixed to be $1$.
The jump is chosen so that $\varepsilon = 1$ in the spherical annulus
between radius $1$ and $2$, as well as between radius $3$ and $4$.
The jump appears between radius $2$ and $3$ and ranges from $10^{-6}$
to $10^6$.  The results confirm that the proposed precondtioners are
robust with respect to jumps in $\varepsilon$.

\begin{table}[h!]
\caption{Iteration counts for test problem using Mesh 3 with $\tau = 0.1$, $\mu^{-1} = 1$, and jumps in $\varepsilon$.}
\begin{center}
\begin{tabular}{|c||c c c c c c c|}
\hline 
			   & $10^{-6}$ & $10^{-4}$ & $10^{-2}$ & $1$ & $10^2$ & $10^4$ & $10^6$ \\ \hline
$\preconW_{\prediag}$                          & 	    28     &   28  	& 	27 	   & 25   &     27     &    21   & 16  \\
$\preconW_\prelow$ 			    & 	    9       &   9  	& 	8 	   & 7     &      7      &    9     & 8  \\
$\preconW_\preupp$ 			    & 	    9       &   9  	& 	8 	   & 8     &      7      &    6     & 6  \\ \hline
$\preconBF_{\prelow\prediag}$              & 	     7      &    8	& 	7	   & 6    &       6      &    8     &  8  \\
$\preconBF_{\prediag\preupp}$             & 	     7      &    7	& 	6 	   & 6     &      6      &    5     & 5  \\
$\preconBF_{\prelow\prediag\preupp}$  &   4         &    4	& 	4	   & 4     &      4      &    4     & 4  \\
\hline 
\end{tabular}
\end{center}
\label{tab:jump-varepsilon}
\end{table}%

Table \ref{tab:jump-mu} reports similar results for jumps in
$\mu^{-1}$, but with $\varepsilon$ is fixed to be $1$.  Similarly to
the previous case, the jump appears between radius $2$ and $3$ and
ranges from $10^{-6}$ to $10^6$.  Outside this region, $\mu^{-1} =
1$. The results show that the proposed precondtioners are also robust with respect to jumps in $\mu^{-1}$.

\begin{table}[h!]
\caption{Iteration counts for test problem using Mesh 3 with $\tau = 0.1$, $\varepsilon = 1$, and jumps in $\mu^{-1}$.}
\begin{center}
\begin{tabular}{|c||c c c c c c c|}
\hline 
			   & $10^{-6}$ & $10^{-4}$ & $10^{-2}$ & $1$ & $10^2$ & $10^4$ & $10^6$ \\ \hline
$\preconW_\prediag$ 			    & 	 17    &   22  	& 	27 	   & 25   &      25    &   25   & 25  \\
$\preconW_\prelow$ 			    & 	 10    &   10  	& 	9 	   & 7     &      7      &   7     & 7  \\
$\preconW_\preupp$ 			    & 	 9    &      9       & 	8 	   & 8     &      8      &   8     & 8  \\ \hline
$\preconBF_{\prelow\prediag}$ 	    & 	 9      &    9	& 	8	   & 6     &      6      &   6     & 6  \\
$\preconBF_{\prediag\preupp}$ 	    & 	 6      &    6	& 	6 	   & 6     &      6      &   6     & 6  \\
$\preconBF_{\prelow\prediag\preupp}$ &   5      &    5	& 	4	   & 4     &      4      &   4     & 4  \\
\hline 
\end{tabular}
\end{center}
\label{tab:jump-mu}
\end{table}%

Analyzing the results in Tables \ref{tab:W-krylov}--\ref{tab:jump-mu}, we
see that the block preconditoners based on exact block factorization
perform slightly better than the block preconditioners based on
well-posedness in terms of iteration count.  The dominant cost in computing the action
of each of these
preconditioners, however, is in approximately solving the
diagonal blocks.  Since such components are present in all of the
preconditioners tested, the overall computational work of applying
each of them is
similar.  Figure \ref{fig:timings} confirms this result when comparing
the timing to completely solve the system over 20 time steps on the
finest grid, Mesh 4, with $\tau = 0.1$ (again assuming $\varepsilon=\mu^{-1}=1$).
Overall, using $\preconBF_{\prelow\prediag\preupp}$
yields the most efficient results.

\begin{figure}[h!]
\centering
\includegraphics[scale=0.3]{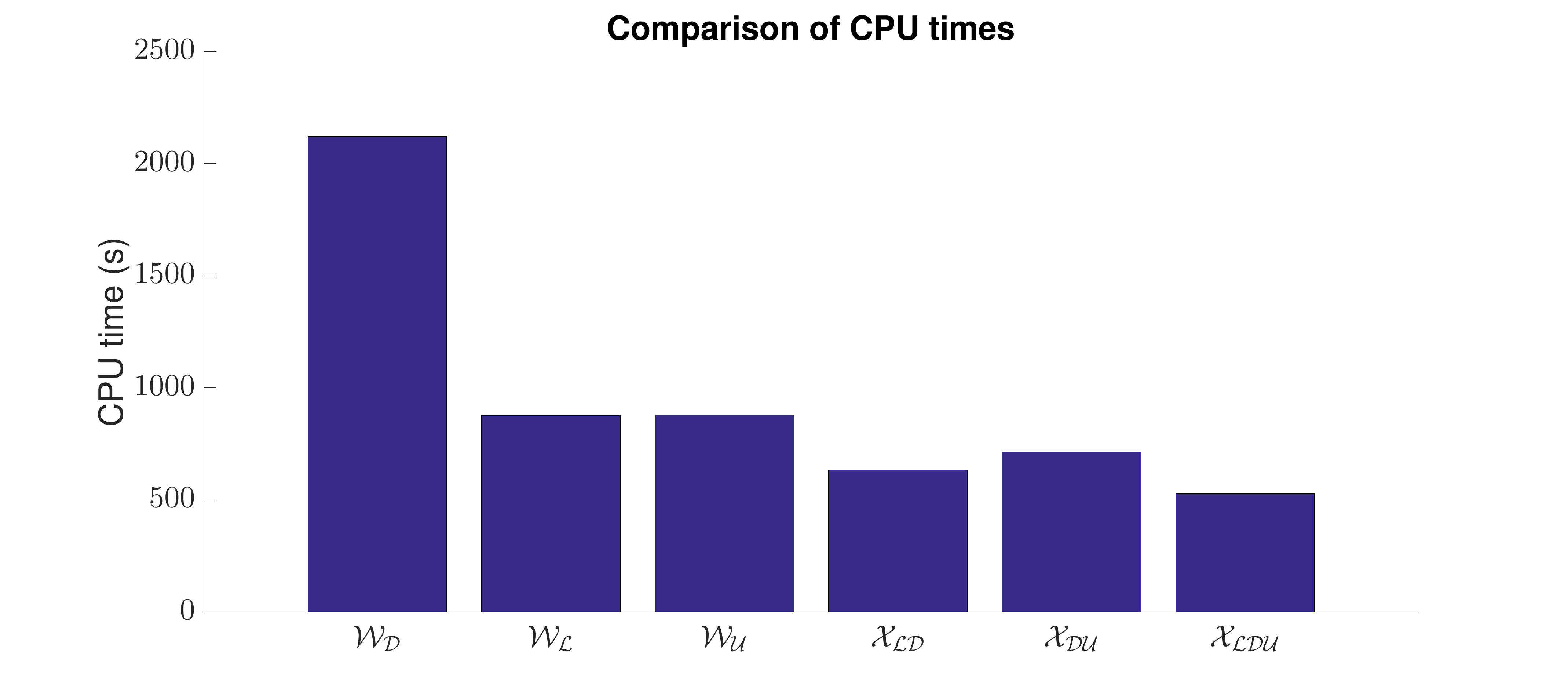}
\caption{Comparison of CPU times using the six different block
  preconditioners for the full simulation of
  \eqref{eq:BcE0}--\eqref{eq:boundary-condition}.  In all runs, $\tau
  = 0.1$, $\varepsilon=\mu^{-1}=1$, and we solve on Mesh 4.}
\label{fig:timings}
\end{figure}

\section{Conclusions}\label{sec:conclude}
In \cite{2013AdlerJ_PetkovV_ZikatanovL-aa}, it was shown that a
structure-preserving discretization of the full time-dependent
Maxwell's equations is capable of resolving the numerical approximation of ADS.
Here, we show that the resulting linear systems are also solved
efficiently.  Block preconditioners for GMRES based on either the well-posedness
of the discretization or on a block factorization approach yield
linear solvers that are robust with respect to simulation parameters,
including time step size and mesh size, as well as the physical
parameters of the problem.  In the process, we have additionally shown
the well-posedness of the structure-preserving discretization and how
to preserve the divergence-free constraint for the magnetic field
within the linear solver itself.  

Such block preconditioners can be applied to other systems, including
those discretized with high-order finite elements which are part of a
deRham complex. Future work involves extending these results to other
applications for which exponentially-decaying solutions exist.  By
using symplectic time integration and structure-preserving
discretizations, we will apply the ideas developed here to build block
preconditioners that will efficiently solve for the solutions that
preserve important physical properties.

\

\textbf{Acknowledgements.}
Ludmil Zikatanov gratefully acknowledges the support for this work from the Department of Mathematics at Tufts University.

\bibliographystyle{plain}
\bibliography{bib_2015_NSF,bib_symplectic}

\begin{thebibliography}{10}

\bibitem{2013AdlerJ_PetkovV_ZikatanovL-aa}
J.~H. Adler, V.~Petkov, and L.~T. Zikatanov.
\newblock Numerical approximation of asymptotically disappearing solutions of
  maxwell's equations.
\newblock {\em SIAM J. Sci. Comput.}, 35(5):S386--S401, 2013.

\bibitem{Benzi.M;Golub.G2005}
M.~Benzi and G.~H. Golub.
\newblock {A preconditioner for generalized saddle point problems}.
\newblock {\em SIAM J. Matrix Anal. Appl.}, 26(1):20--41, 2005.

\bibitem{Benzi.M;Golub.G;Liesen.J2005}
M.~Benzi, G.~H. Golub, and J.~Liesen.
\newblock Numerical solution of saddle point problems.
\newblock {\em Acta Numer.}, 14:1--137, 2005.

\bibitem{2011ColombiniF_PetkovV_RauchJ-aa}
F.~Colombini, V.~Petkov, and J.~Rauch.
\newblock Incoming and disappearing solutions for {M}axwell's equations.
\newblock {\em Proc. Amer. Math. Soc.}, 139(6):2163--2173, 2011.

\bibitem{2014ColombiniF_PetkovV_RauchJ-aa}
F.~Colombini, V.~Petkov, and J.~Rauch.
\newblock Spectral problems for non-elliptic symmetric systems with dissipative
  boundary conditions.
\newblock {\em J. Funct. Anal.}, 267(6):1637--1661, 2014.

\bibitem{Cyr.E;Shadid.J;Tuminaro.R.2013a}
E.~C. Cyr, J.~N. Shadid, and R.~S. Tuminaro.
\newblock {A new approximate block factorization preconditioner for
  two-dimensional incompressible (reduced) resistive MHD}.
\newblock {\em SIAM J. Sci. Comput.}, 35(3):701--730, 2013.

\bibitem{2006ElmanH_HowleV_ShadidJ_ShuttleworthR_TuminaroR-aa}
H.~C. Elman, V.~E. Howle, J.~N. Shadid, R.~Shuttleworth, and R.~S. Tuminaro.
\newblock Block preconditioners based on approximate commutators.
\newblock {\em SIAM J. Sci. Comput.}, 27(5):1651--1668, 2006.

\bibitem{Elman.H;Silvester.D;Wathen.A.2005a}
H.~C. Elman, D.~J. Silvester, and A.~J. Wathen.
\newblock {\em Finite elements and fast iterative solvers: with applications in
  incompressible fluid dynamics}.
\newblock OUP Oxford, 2005.

\bibitem{1985FengK-aa}
K.~Feng.
\newblock On difference schemes and symplectic geometry.
\newblock In {\em Proceedings of the 1984 {B}eijing symposium on differential
  geometry and differential equations}, pages 42--58. Science Press, Beijing,
  1985.

\bibitem{1986FengK-aa}
K.~Feng.
\newblock Difference schemes for {H}amiltonian formalism and symplectic
  geometry.
\newblock {\em J. Comput. Math.}, 4(3):279--289, 1986.

\bibitem{2010FengK_QinM-aa}
K.~Feng and M.~Z. Qin.
\newblock {\em Symplectic geometric algorithms for {H}amiltonian systems}.
\newblock Zhejiang Science and Technology Publishing House, Hangzhou; Springer,
  Heidelberg, 2010.
\newblock Translated and revised from the Chinese original, With a foreword by
  Feng Duan.

\bibitem{1990FengK_WuH_QinM-aa}
K.~Feng, H.~M. Wu, and M.~Z. Qin.
\newblock Symplectic difference schemes for linear {H}amiltonian canonical
  systems.
\newblock {\em J. Comput. Math.}, 8(4):371--380, 1990.

\bibitem{Hiptmair.R;Xu.J2007}
R.~Hiptmair and J.~Xu.
\newblock Nodal auxiliary space preconditioning in {${\bf H}({\bf curl})$} and
  {${\bf H}({\rm div})$} spaces.
\newblock {\em SIAM J. Numer. Anal.}, 45(6):2483--2509, 2007.

\bibitem{Hu.K;Ma.Y;Xu.J.2014a}
K.~Hu, Y.~Ma, and J.~Xu.
\newblock {Stable finite element methods preserving $\nabla \cdot \bm{B} = 0$
  exactly for MHD models}.
\newblock Submitted to Numerische Mathematik, 2014.

\bibitem{klawonn1998block}
A.~Klawonn.
\newblock {Block-triangular preconditioners for saddle point problems with a
  penalty term}.
\newblock {\em SIAM Journal on Scientific Computing}, 19:172, 1998.

\bibitem{Loghin.D;Wathen.A2004}
D.~Loghin and A.~J. Wathen.
\newblock Analysis of preconditioners for saddle-point problems.
\newblock {\em SIAM J. Sci. Comput.}, 25(6):2029--2049, 2004.

\bibitem{Hu.K;Hu.X;Xu.J;Ma.Y.2014a}
Y.~Ma, K.~Hu, X.~Hu, and J.~Xu.
\newblock Robust preconditioners for incompressible mhd models.
\newblock {\em arXiv preprint arXiv:1503.02553}, 2015.

\bibitem{1974MajdaA-aa}
A.~Majda.
\newblock Disappearing solutions for the dissipative wave equation.
\newblock {\em Indiana Univ. Math. J.}, 24(12):1119--1133, 1974/75.

\bibitem{1976MajdaA-aa}
A.~Majda.
\newblock The location of the spectrum for the dissipative acoustic operator.
\newblock {\em Indiana Univ. Math. J.}, 25(10):973--987, 1976.

\bibitem{Mardal.K;Winther.R2004}
K.~A. Mardal and R.~Winther.
\newblock Uniform preconditioners for the time dependent {S}tokes problem.
\newblock {\em Numer. Math.}, 98(2):305--327, 2004.

\bibitem{Mardal.K;Winther.R2010}
K.~A. Mardal and R.~Winther.
\newblock Preconditioning discretizations of systems of partial differential
  equations.
\newblock {\em Numer. Linear Algebra Appl.}, 2010.

\bibitem{1989PetkovV-aa}
V.~Petkov.
\newblock {\em Scattering theory for hyperbolic operators}, volume~21 of {\em
  Studies in Mathematics and its Applications}.
\newblock North-Holland Publishing Co., Amsterdam, 1989.

\bibitem{2013PetkovV-aa}
V.~Petkov.
\newblock Scattering problems for symmetric systems with dissipative boundary
  conditions.
\newblock In {\em Studies in phase space analysis with applications to {PDE}s},
  volume~84 of {\em Progr. Nonlinear Differential Equations Appl.}, pages
  337--353. Birkh\"auser/Springer, New York, 2013.

\bibitem{Phillips.E;Elman.H;Cyr.E;Shadid.J;Pawlowski.R.2014a}
E.~G. Phillips, H.~C. Elman, E.~C. Cyr, J.~N Shadid, and R.~P. Pawlowski.
\newblock A block preconditioner for an exact penalty formulation for
  stationary mhd.
\newblock {\em {SIAM} J. Sci. Comput.}, 36(6):B930--B951, 2014.

\bibitem{Rusten.T;Winther.R1992}
T.~Rusten and R.~Winther.
\newblock A preconditioned iterative method for saddlepoint problems.
\newblock {\em SIAM J. Matrix Anal. Appl.}, 13(3):887--904, 1992.
\newblock Iterative methods in numerical linear algebra (Copper Mountain, CO,
  1990).

\bibitem{Schoeberl2007}
J.~Sch\"{o}berl and W.~Zulehner.
\newblock {Symmetric indefinite preconditioners for saddle point problems with
  applications to {PDE}-constrained optimization problems}.
\newblock {\em SIAM J. Matrix Anal. Appl.}, 29(3):752----773, 2007.

\bibitem{Vassilevski.P2008}
P.~S. Vassilevski.
\newblock {\em Multilevel block factorization preconditioners}.
\newblock Springer, New York, 2008.
\newblock Matrix-based analysis and algorithms for solving finite element
  equations.

\end{thebibliography}

\end{document}